\documentclass{article}[12pt]
\usepackage[utf8]{inputenc}

\usepackage{amssymb,amsmath,amsfonts,amsthm,array,comment,mathrsfs,times,graphicx}
\usepackage[frenchb,english]{babel}
\usepackage[frenchb]{babel}
\usepackage{color,soul}
\usepackage[new]{old-arrows}
\usepackage{enumitem}
\usepackage{makeidx}
\makeindex
\usepackage{geometry}
\usepackage{longtable,tabularx}
\usepackage{amsmath}
\usepackage[all,cmtip]{xy}
\usepackage{tikz-cd}
 
\newtheorem{theorem}{Theorem}[section]
\newtheorem{coro}[theorem]{Corollary}
\newtheorem{lemm}[theorem]{Lemma}
\newtheorem{prop}[theorem]{Proposition}

\newtheorem{conjecture}[theorem]{Conjecture}

\newtheorem{theo}[theorem]{Theorem}

\theoremstyle{definition}
\newtheorem{defi}[theorem]{Definition}
\newtheorem{rema}[theorem]{Remark}

\newtheorem{exam}[theorem]{Example}
\newtheorem{thmx}{Theorem}

\newcommand{\el}{\mathrm{el}}
\newcommand{\ab}{\mathrm{ab}}
\newcommand{\Q}{\mathbb{Q}}
\newcommand{\Z}{\mathbb{Z}}
\newcommand{\Ram}{\mathrm{Ram}}
\DeclareMathOperator{\Ind}{Ind}
\DeclareMathOperator{\Inf}{Inf}
\DeclareMathOperator{\Res}{Res}

\title{On the Galois structure of units in totally real $p$-rational number fields}

\author{Zakariae Bouazzaoui, Donghyeok Lim}
\date{}
 
\begin{document}
 
\maketitle

\begin{abstract} 
The theory of factor-equivalence of integral lattices establishes a far-reaching relationship between the Galois module structure of the unit group of the ring of integers of a number field and its arithmetic. For a number field $K$ that is Galois over $\mathbb{Q}$ or an imaginary quadratic field, we prove a necessary and sufficient condition on the quotients of class numbers of subfields of $K$, for the quotient $E_{K}$ of the unit group of the ring of integers of $K$ modulo the subgroup of roots of unity to be factor equivalent to the standard cyclic Galois module. Using strong arithmetic properties of totally real $p$-rational number fields, we prove that the non-abelian $p$-rational $p$-extensions of $\mathbb{Q}$ do not admit Minkowski units, thereby extending a result of Burns to non-abelian number fields. We also study the relative Galois module structure of $E_{L}$ for varying Galois extensions $L/F$ of totally real $p$-rational number fields whose Galois groups are isomorphic to a fixed finite group $G$. In that case, we prove that there exists a finite set $\Omega$ of $\mathbb{Z}_p[G]$-lattices such that for every $L$, $\mathbb{Z}_{p} \otimes_{\mathbb{Z}} E_{L}$ is factor equivalent to $\mathbb{Z}_{p}[G]^{n} \oplus X$ as $\mathbb{Z}_p[G]$-lattices for some $X \in \Omega$ and an integer $n \geq 0$.
\end{abstract}


\section{Introduction}

Let $k$ be a number field that is either $\mathbb{Q}$ or an imaginary quadratic number field. For brevity, let us call such a number field \textit{admissible} (cf. \cite{Burns2}). Let $K$ be a Galois extension of $k$ with Galois group $G_{K/k}$ such that the infinite places of $k$ are unramified in $K$. Let $\mathcal{O}_K$ be its ring of integers. In the late nineteenth century, Minkowski proved that there exists a unit in $\mathcal{O}_K$ whose conjugates over $k$ generate a full-rank subgroup of $\mathcal{O}_{K}^{\times}$. As a consequence of Minkowski's theorem, if we regard $\mathcal{O}_{K}^{\times}$ as a $\mathbb{Z}[G_{K/k}]$-module, then there is an isomorphism of $\mathbb{Q}[G_{K/k}]$-modules
\begin{equation*}
    \mathbb{Q} \otimes_{\mathbb{Z}} \mathcal{O}_{K}^{\times} \simeq A_{G_{K/k}} := \mathbb{Q}[G_{K/k}]/(s_{G_{K/k}}).
\end{equation*}
For a finite group $G$, we write $s_G$ for the element $\sum_{g\in G}g$ in $\mathbb{Z}[G] \subset\mathbb{Q}[G]$, and write $(s_G)$ for the submodule generated by $s_G$. 

Despite the long history of interest in $\mathcal{O}^{\times}_K$, the $\mathbb{Z}[G_{K/k}]$-module structure of $\mathcal{O}^{\times}_{K}$ is barely understood until now. One of the principal problems in this area is to determine if $K$ has a (strong) \textit{Minkowski unit}, in other words, a unit $u \in \mathcal{O}^{\times}_K$ whose conjugates over $k$ generate $\mathcal{O}_{K}^{\times}$ modulo the subgroup $\mu(K)$ of roots of unity of $K$. The problem is equivalent to asking whether the \textit{unit lattice} $E_{K}:= \mathcal{O}_{K}^{\times}/\mu(K)$ is isomorphic to the \textit{standard cyclic Galois module} $\mathcal{A}_{G_{K/k}}:=\mathbb{Z}[G_{K/k}]/(s_{G_{K/k}})$ as $\mathbb{Z}[G_{K/k}]$-modules.

Studying the Galois module structure of the unit lattice is complicated for the following reasons. First of all, it is very difficult to compute a system of fundamental units. Therefore, number theorists tried to use arithmetic information on $K$ to study the cyclicity of $E_{K}$ as a $\mathbb{Z}[G_{K/k}]$-module. However, the involved arithmetic information is still difficult to obtain. Lastly, to be sure that certain arithmetic information suffices to guarantee that $E_{K}$ is cyclic, we need enough results on the classification of $\mathbb{Z}[G_{K/k}]$-lattices of the same rank as $E_K$. However, even for a finite group $G$ with a simple structure, classifying integral representations of $G$ is very difficult. Hence, studying the Galois module structure of the unit lattice requires knowledge of both the theory of integral representations of finite groups and the arithmetic of number fields. We examine both aspects of the problem in this paper, which is organized into $8$ sections.

\vskip 13pt
\begin{center}
\noindent I. Representation-Theoretic Part (\S \ref{sec-genus} -- \S \ref{Proof of the Main Theorem})
\end{center}
\vskip 5pt

Several approaches have been introduced to study the integral Galois module structure of $E_{K}$. One approach is to study whether $\mathbb{Z}_p \otimes_{\mathbb{Z}} E_{K}$ is isomorphic to $\mathbb{Z}_{p}[G_{K/k}]/(s_{G_{K/k}}) \simeq \mathbb{Z}_p \otimes_{\mathbb{Z}} \mathcal{A}_{G_{K/k}}$ as $\mathbb{Z}_{p}[G_{K/k}]$-modules for every prime $p$ (the problem of existence of \textit{local} Minkowski units), by using the Krull-Schmidt theorem. This approach has limitations because there are many non-isomorphic indecomposable $\mathbb{Z}_p[G]$-lattices even for finite groups $G$ with simple structures. In this context, recall that two $\mathbb{Z}[G]$-lattices $M$ and $N$ are said to be genus equivalent if we have $\mathbb{Z}_p[G]\otimes_{\mathbb{Z}} M \simeq \mathbb{Z}_p[G]\otimes_{\mathbb{Z}} N$ for every prime $p$.

In the $1980$s, A. Fr$\ddot{\textrm{o}}$hlich and his students began to develop the theory of \textit{factor equivalence} between integral lattices. This approach is computationally more effective because instead of studying each lattice, it focuses on comparing two lattices from the start. Although factor equivalence is weaker than genus equivalence, it nevertheless has profound applications in number theory, explaining how arithmetic influences the Galois module structure of the unit lattice for general Galois groups. Moreover, it was fruitful in yielding a necessary and sufficient condition for the existence of local Minkowski units, valid for all abelian extensions \cite{Burns2}.

In the 2010s, A. Bartel revisited the connection between the integral Galois module structure of unit lattices and the arithmetic of number fields. In \cite{Bartel12}, Bartel used the \textit{regulator constant} of integral lattices (cf. \cite{dokchitser2009regulator}) to obtain the most general Brauer-Kuroda type formula for dihedral extensions of number fields of degree $2p$ for an odd prime $p$. The regulator constant establishes a relationship between the Galois module structure of the unit lattice of a number field and a quotient of Dirichlet regulators of its subfields. In fact, the factor equivalence and the regulator constants are closely related. This connection is made precise by a theorem of Bartel \cite[Cor. 2.12.]{Bartel14}, which states that for a finite group $G$, two $\mathbb{Z}[G]$ lattices $\mathcal{M}$ and $\mathcal{N}$ with isomorphic self-dual rational representations are factor equivalent if and only if their regulator constants coincide for all $G$-relations (for the precise definition of $G$-relations, see Definition \ref{G-relation} and Example \ref{exam-G-relation}).\\

 In the first part of this paper, we establish a formula for the regulator constants of $\mathcal{A}_{G}=\mathbb{Z}[G]/(s_G)$ (Proposition \ref{Regulator constant C(Z[G]/G)}), which yields the following theorem.

\begin{thmx}\label{maintheorem}
Let $K$ be a number field that is Galois over an admissible field $k$ with Galois group $G_{K/k}$. Assume that $K/k$ is unramified at the infinite places of $k$. For each subgroup $H$ of $G_{K/k}$, let $h_{K^{H}}$ (resp. $w_{K^{H}}$) denote the class number (resp. the number of roots of unity) of the fixed field $K^{H}$. Define $\lambda(H)$ to be the order of the kernel of the map $$H^{1}(H,\mu(K)) \to H^{1}(H,\mathcal{O}_{K}^{\times})$$ induced by the embedding $\mu(K) \hookrightarrow \mathcal{O}_{K}^{\times}$. Then, $E_K$ is factor equivalent to $\mathcal{A}_{G_{K/k}}$ if and only if the equality
\begin{equation*}
\underset{H \leq G_{K/k}}{\prod} \bigg ( \frac{|H| \cdot h_{K^{H}} \cdot \lambda(H)}{w_{K^{H}}} \bigg)^{n_H} = 1
\end{equation*}
holds for every $G$-relation $\sum_{H \leq G_{K/k}} n_H H$ of $G_{K/k}$.
\end{thmx}

Theorem \ref{maintheorem} generalizes the previously known necessary conditions on quotients of class numbers of subfields for the existence of Minkowski units, extending them to all non-cyclic finite groups.

\vskip 13pt
\begin{center}
 II. Arithmetic Part (\S \ref{sec-p-rational} -- \S \ref{relGalmod})
\end{center}
\vskip 5pt

Even for number fields of small degree, it is very difficult to obtain detailed information on unit lattices, since the relevant arithmetic invariants are challenging to analyze. In \cite{Burns2}, Burns exploited strong arithmetic properties of $p$-power genus field extensions to study the existence of local Minkowski units in abelian $p$-extensions of admissible fields. In the second part of this work, we investigate the Galois module structure of the unit lattices for a new family of number fields, namely totally real \textit{$p$-rational} number fields.

Let $p$ be an odd prime. A number field $F$ is called $p$-rational if the Galois group of the maximal pro-$p$ extension $F_{S_p}$ of $F$ that is unramified outside the set $S_p$ of $p$-adic primes is a free pro-$p$ group (cf. \cite{Mova-Do,Greenberg}). Totally real $p$-rational number fields form an interesting class to work with. Their $p$-class numbers can be computed rather explicitly, a feature that is not often available in general and which makes it possible to carry out our theory in practice. Moreover, they give rise to infinitely many infinite non-abelian pro-$p$ towers, which not only allow us to study non-abelian cases but also provide a wealth of examples to which our results apply. In \S \ref{Factor equivalence class of the unit groups of $p$-rational}, we establish the following theorem on the non-existence of Minkowski units.

\begin{thmx}\label{maintheorem2}
Let $p$ be an odd prime. If $F$ is a non-abelian $p$-rational $p$-extension of $\mathbb{Q}$, then $F$ does not have a local Minkowski unit.
\end{thmx}
Theorem \ref{maintheorem2} provides an infinite family of non-abelian number fields without Minkowski units. This result may be viewed as a non-abelian extension of Burns's theorem on the existence of local Minkowski units in $p$-power genus field extensions of admissible number fields (cf. \cite[Theorem 5]{Burns2}).

We also study in \S \ref{relGalmod} the relative Galois module structure of unit lattices for \textit{varying} Galois extensions of totally real $p$-rational number fields with Galois group isomorphic to a \textit{fixed} finite group $G$. Note that in this setting the base field is not required to be admissible. Since the $\mathbb{Z}$-rank of the unit lattice is unbounded as the extension varies, infinitely many non-isomorphic indecomposable $\mathbb{Z}_p[G]$-lattices can occur in the Krull-Schmidt decomposition of its $p$-adic completion.

In \cite{Burns4}, Burns established that for every finite group $G$ and a finite set $S$ of primes of $\mathbb{Z}$ containing $p$, there exists a natural infinite family of relative Galois extensions $\mathscr{L}/\mathscr{K}$ with $\mathrm{Gal}(\mathscr{L}/\mathscr{K}) \simeq G$ in which the sum of the $\mathbb{Z}_{p}$-ranks of the non-projective indecomposable components in a Krull-Schmidt decomposition of $\mathbb{Z}_p \otimes_{\mathbb{Z}} E_{\mathscr{L},S}$ (as a $\mathbb{Z}_p[G]$-lattice) is uniformly bounded. Here, $E_{\mathscr{L},S}$ denotes the quotient of $S$-unit group $\mathcal{O}_{\mathscr{L},S}^{\times}$ of $\mathscr{L}$ by $\mu(\mathscr{L})$. It then follows from the Jordan-Zassenhaus theorem that only finitely many non-isomorphic indecomposable $\mathbb{Z}_p[G]$-lattices appear in the Krull-Schmidt decomposition of $\mathbb{Z}_p \otimes_{\mathbb{Z}} E_{\mathscr{L},S}$ for such extensions $\mathscr{L}/\mathscr{K}$ belonging to this family.

In \S \ref{relGalmod}, we observe a similar phenomenon in the relative Galois module structure of the group of \textit{ordinary} units when the number fields are totally real and $p$-rational.

\begin{thmx}\label{maintheorem3}
Let $G$ be a finite group. Then, there exists a finite set $\Omega$ of $\mathbb{Z}_p[G]$-lattices such that for every relative Galois extension $\mathscr{L}/\mathscr{K}$ of totally real $p$-rational number fields with $\mathrm{Gal}(\mathscr{L}/\mathscr{K}) \simeq G$, there exists $X \in \Omega$ and a non-negative integer $m$ such that $\mathbb{Z}_p \otimes_{\mathbb{Z}} E_{\mathscr{L}}$ is factor equivalent to $X \oplus \mathbb{Z}_{p}[G]^{m}$ as $\mathbb{Z}_{p}[G]$-lattices.
\end{thmx}

The arithmetic of a totally real $p$-rational number field becomes particularly simple when there is a unique $p$-adic prime (cf. \S \ref{uniquepprime}). In this case, Theorem \ref{maintheorem3} can be sharpened. We now make this setting more precise.

Let $p$ be an odd prime, and let $F$ be a totally real $p$-rational number field. 
Fix a finite non-$p$-adic prime $\mathfrak{q}$ of $F$ that does not split in the cyclotomic $\mathbb{Z}_p$-extension $F_{\infty}/F$. 
Denote by $F_{S_p \cup \{\mathfrak{q}\}}$ (resp. $F_{\{\mathfrak{q}\}}$) the maximal pro-$p$ extension of $F$ unramified outside $S_p \cup \{\mathfrak{q}\}$ (resp. unramified outside $\mathfrak{q}$). We remark that $\mathrm{Gal}(F_{S_p \cup \{\mathfrak{q}\}}/F)$ is a Demu\v{s}kin group of rank $2$, while $F_{\{\mathfrak{q}\}}/F$ is finite. For a finite group $G$, we denote by $I_G$ the augmentation ideal of the group ring $\mathbb{Z}[G]$.

\begin{thmx}\label{maintheorem4}
 Suppose that $p$ does not split in $F_{S_p \cup \{\mathfrak{q}\}}$. Then, for every Galois extension $\mathscr{L}/\mathscr{K}$ of number fields satisfying $$F_{\{\mathfrak{q}\}} \subseteq \mathscr{K} \subseteq \mathscr{L} \subset F_{S_p \cup \{\mathfrak{q}\}}$$ with Galois group $G$, the lattice $E_{\mathscr{L}}$ is factor equivalent to  
 $$\mathcal{A}_{G} \oplus I_{G} \oplus \mathbb{Z} \oplus \mathbb{Z}[G]^{[\mathscr{K}:\mathbb{Q}]-2}$$ as $\mathbb{Z}[G]$-lattices.
\end{thmx}

\vskip 15pt
\begin{center}
 \textbf{Notations}
\end{center}

For a finite group $G$, let $s_G := \sum_{g \in G} g$ denote the trace element, $I_G$ the augmentation ideal, and $\mathcal{A}_G := \mathbb{Z}[G]/(s_G)$ the standard cyclic $\mathbb{Z}[G]$-module. We write $G^{\mathrm{ab}}$ for the abelianization of $G$, and $(G:H)$ for the index of $H$ in $G$ for every subgroup $H$ of $G$. For an abelian group $A$, $\mathrm{rk}_p(A)$ denotes its $p$-rank. If $X$ is a $G$-set, $X^G$ is the subset fixed by $G$. For a natural number $n$, $|n|_p$ denotes its $p$-part, and for $x \in \mathbb{Q}^\times$, $v_p(x)$ denotes the $p$-adic valuation. Finally, all modules over a ring $R$ are understood to be left $R$-modules.\\

For an extension $L/K$ of number fields, we write $\Ram(L/K)$ for the set of places of $K$ that ramify in $L$. If $L/K$ is Galois, we denote $G_{L/K} := \mathrm{Gal}(L/K)$. For a number field $F$, we let $h_F$ denote its class number, $w_F$ the number of roots of unity in $F$, and $R_F$ its Dirichlet regulator. If $v$ is a place of $F$, then $F_v$ denotes the completion of $F$ at $v$.\\

In the sections devoted to $p$-rationality, we adopt the following notation. We fix an odd prime $p$ and denote by $\mathfrak{h}_F$ the $p$-class number of a number field $F$. We write $F_{\infty}$ for its cyclotomic $\mathbb{Z}_p$-extension, $F_n$ for the $n$-th layer of $F_{\infty}/F$, and $H_F$ for the Hilbert $p$-class field of $F$. We also let $S_p$ denote the set of $p$-adic places of a number field. Since the base field will always be clear from the context, this notation will cause no ambiguity.

In addition, we introduce the following notation, which will be used frequently in \S \ref{uniquepprime} and \S \ref{Factor equivalence class of the unit groups of $p$-rational}. 
For a finite set $S$ of places of $F$, we write $F_S$ for the maximal pro-$p$ extension of $F$ unramified outside $S$. 
If $\mathfrak{q}$ is a non-$p$-adic prime of $F$, we set:

\begin{tabbing}
		$\mathfrak{q}_L$ \quad\,\,\,\= : the unique prime above $\mathfrak{q}$ in an extension $F\subset L\subset F_{S_p \cup \{\mathfrak{q}\}}$
        (when $F$ is $p$-rational and\\
        \> \; $S_p \cup \{\mathfrak{q}\}$ is primitive for $(F,p)$, see Lemma \ref{nonsplitting});\\
		$L^{\el}_{p,\mathfrak{q}}$ \> : the maximal elementary abelian $p$-extension of $L$ contained in $F_{S_p \cup \{\mathfrak{q}\}}$ (in the same setting as above);\\
        $p_L$ \> : the unique prime of $L$ above $p$ for $F\subset L\subset F_{S_p \cup \{\mathfrak{q}\}}$ (when $p$ does not split in $F_{S_p \cup \{\mathfrak{q}\}}$);\\
        $\mathcal{I}_{L,v}$ \> : the inertia subgroup of $G_{L^{\el}_{p, \mathfrak{q}}/L}$ at $v$ (see the discussion 
preceding Lemma \ref{lemma-application of Bouc}).
\end{tabbing}

\vskip 15pt
\noindent \textbf{Acknowledgements}

\vskip 5pt
\noindent We would like to thank David Burns for his helpful discussion and encouragement, Abbas Movahhedi for pointing out an error on $p$-rationality in the first version of the paper, and Christian Maire for Remark \ref{infinitude pq}. We also would like to thank Jilali Assim, Zouhair Boughadi, Asuka Kumon, El Boukhari Saad, Bouchaïb Sodaïgui, and Youness Mazigh for their interest in this work and helpful comments. The revision of this paper was carried out while D. Lim was a visiting researcher at FEMTO-ST in Besan\c{c}on in 2023 and during his visit to the Institute for Advanced Studies in Mathematics (IASM) in Harbin Institute of Technology in 2025. D. Lim would like to thank both institutions for their friendly environment. Finally, we would like to thank the anonymous referee for his or her patience over a long period and for the sharp and constructive comments, which greatly enhanced the overall quality of the paper. In particular, we are grateful for the valuable mathematical suggestions regarding Section~4.

 \vskip 10pt
\noindent \textbf{Funding}
\vskip 5pt
\noindent 
D. Lim was supported by the Basic Science Research Program through the National Research
Foundation of Korea (NRF) funded by the Ministry of Education (Grant No. NRF-2022R1I1A1A01071431). He was also supported by the Core Research Institute Basic Science Research Program
through the National Research Foundation of Korea(NRF), funded by the Ministry of Education (Grant No. 2019R1A6A1A11051177).

\section{Genus equivalence}\label{sec-genus}

Throughout this section, let \( k \) be an admissible field and \( K/k \) a Galois extension of number fields in which the infinite places of \( k \) are unramified. Since the \( \mathbb{Z}[G_{K/k}] \)-module structure of \( E_K \) is difficult to study (cf. \cite{Brumer, Marko}), it is natural to first examine the \( \mathbb{Z}_p[G_{K/k}] \)-module structure of \( \mathbb{Z}_p \otimes_{\mathbb{Z}} E_K \) for all primes \( p \). For a Dedekind domain \( R \) and a finite group $G$, a finitely generated \( R[G] \)-module that is torsion-free as an \( R \)-module is called an \( R[G] \)-\textit{lattice}.

\begin{defi}
Two $R[G]$-lattices $M$ and $N$ are said to be \textit{genus equivalent} if for every non-zero prime $\mathfrak{p}$ of $R$, we have $$R_{\mathfrak{p}} \otimes_R M \simeq R_{\mathfrak{p}} \otimes_R N$$ as $R_{\mathfrak{p}}[G]$-modules, where $R_{\mathfrak{p}}$ denotes the completion of $R$ at $\mathfrak{p}$.
\end{defi}

If $p$ does not divide $|G_{K/k}|$, then we have $$\mathbb{Z}_p \otimes_{\mathbb{Z}} E_{K}  \simeq \mathbb{Z}_p \otimes_{\mathbb{Z}} \mathcal{A}_{G_{K/k}}$$ by representation theory (cf. \cite[\S 15.5]{serre1977linear}). Therefore, the study of the genus equivalence class of $E_{K}$ concerns only those primes $p$ dividing $|G_{K/k}|$.\\ 
A $\mathbb{Z}_{p}[G]$-lattice is said to be \textit{indecomposable} if it is not a direct sum of two proper $\mathbb{Z}[G]$-sublattices. The Krull-Schmidt theorem is available for $\mathbb{Z}_p[G]$-lattices.

\begin{theo}(cf. \cite[p. 83]{curtis1966representation})
Let $M$ be a $\mathbb{Z}_{p}[G]$-lattice, and suppose that
\begin{equation*}
    M \simeq U_{1} \oplus \cdots \oplus U_{m} \simeq V_{1} \oplus \cdots \oplus V_{n}
\end{equation*}
are two decompositions of $M$ into indecomposable $\mathbb{Z}_{p}[G]$-sublattices. Then, we have $m=n$, and after a suitable reindexing we have $U_{i} \simeq V_{i}$ as $\mathbb{Z}_{p}[G]$-lattices for every $1 \leq i \leq m$.
\end{theo}
Therefore, if we classify the indecomposable $\mathbb{Z}_{p}[G_{K/k}]$-lattices of $\mathbb{Z}_p$-rank at most $|G_{K/k}|$, then \textit{in principle}, we can study the $\mathbb{Z}_{p}[G_{K/k}]$-module structure of $\mathbb{Z}_p \otimes_{\mathbb{Z}} E_{K}$ by computing the multiplicity of each indecomposable $\mathbb{Z}_{p}[G_{K/k}]$-lattice in the Krull-Schmidt decomposition of $\mathbb{Z}_p \otimes_{\mathbb{Z}} E_K$.

\begin{exam}\label{latticeexample}
\begin{enumerate}
    \item[(i)] If $G$ is cyclic of order $p$, then by a theorem of Diederichsen, there are precisely three isomorphism classes of indecomposable $\mathbb{Z}_{p}[G]$-lattices \cite{Diederichsen}. From this, we can easily check the genus equivalence of $E_{K}$ and $\mathcal{A}_{G_{K/k}}$ for every cyclic extension $K/k$ of prime degree.
    
    \item[(ii)] When $G$ is the elementary abelian $p$-group $(\mathbb{Z}/p\mathbb{Z})^2$ of rank $2$, Payan, Bouvier, and Duval classified indecomposable $\mathbb{Z}_{p}[G]$-lattices that can be realized as $\mathbb{Z}_{p}[G]$-sublattices of $\mathbb{Z}_p \otimes_{\mathbb{Z}} E_{K}$ for some Galois extensions $K/k$ with Galois group $G_{K/k}\simeq G$ (cf. \cite{BouvierPayan, Duval2, duval1981structure}). In \cite{duval1981structure}, Duval used these results to study the genus equivalence of $E_{K}$ and $\mathcal{A}_{G_{K/k}}$ in the case $G_{K/k}\simeq (\mathbb{Z}/p\mathbb{Z})^2$. 
    
    \item[(iii)] Assume that $G_{K/k}$ is a metacyclic group of the form $\mathbb{Z}/p\mathbb{Z} \rtimes T$, where $T$ is a cyclic group whose order divides $p-1$. Marszalek obtained necessary and sufficient conditions for the genus equivalence of $E_{K}$ and $\mathcal{A}_{G_{K/k}}$ for the case when the action of $T$ on $\mathbb{Z}/p\mathbb{Z}$ is faithful by using the classification of integral representations \cite{Mars1,Mars2}. The interested readers can also refer to \cite{Jaulent1,Jaulent2}.

\end{enumerate}
\end{exam}

The study of the genus equivalence class of $E_{K}$ has not been extended to more general Galois groups, because the classification of the indecomposable integral representations is highly complicated, even for groups with simple structures.

\begin{rema}
\begin{enumerate}
    \item[(i)] In \cite{HelRei1, HelRei2}, Heller and Reiner studied the indecomposable $\mathbb{Z}_{p}[\mathbb{Z}/p^{e}\mathbb{Z}]$-lattices for $e \geq 2$. In \cite{HelRei1}, they proved that there are precisely $4p+1$ non-isomorphic indecomposable representations of $\mathbb{Z}/p^{2}\mathbb{Z}$ over $\mathbb{Z}_{p}$. However, there are infinitely many non-isomorphic indecomposable representations over $\mathbb{Z}_{p}$ if $e$ is larger than $2$ \cite{HelRei2}. There has been no attempt to classify indecomposable $\mathbb{Z}_{p}[\mathbb{Z}/p^{3}\mathbb{Z}]$-lattices  even for small $\mathbb{Z}_{p}$-ranks.
    
    \item[(ii)] Duval obtained a necessary and sufficient condition for the genus equivalence between $E_{K}$ and $\mathcal{A}_{G_{K/k}}$ only in the cases $p=2$ or $3$ for $G_{K/k} \simeq (\mathbb{Z}/p\mathbb{Z})^2$, as the classification was too difficult and incomplete for larger $p$. For an illustration of the difficulty of classifying integral representations of $(\mathbb{Z}/p\mathbb{Z})^2$ over $\mathbb{Z}_{p}$, see the table on page 241 of \cite{duval1981structure}, which is far from exhaustive for general $p$.
\end{enumerate}

\end{rema}

\section{Factor equivalence, theorems of Burns and regulator constant}\label{SecFactor}

In this section, we explain basic concepts in the theory of factor equivalence and its application to the Galois module structure of unit lattices. We then present some basic results on the regulator constants. Lastly, we explain the connection between factor equivalence and regulator constants.

\subsection{Factor equivalence of integral lattices}\label{Sec3.1}

Throughout this section, let $K$ be a finite extension of $\mathbb{Q}$ or $\mathbb{Q}_{p}$. We write $\mathcal{O}_K$ for its ring of integers and $\mathrm{Id}_K$ for its group of fractional ideals. For a finite group $G$, denote by $\mathscr{S}(G)$ the set of subgroups of $G$.

As in the theory of genus equivalence, the theory of factor equivalence \cite{Nelson} compares two $\mathcal{O}_{K}[G]$-lattices $\mathcal{M}$ and $\mathcal{N}$ such that we have $K \otimes_{\mathcal{O}_K} \mathcal{M} \simeq K \otimes_{\mathcal{O}_K} \mathcal{N}$ as $K[G]$-modules. From such a  $K[G]$-isomorphism, one obtains an injective $\mathcal{O}_{K}[G]$-module homomorphism $\imath : \mathcal{M} \to \mathcal{N}$. The relation of factor equivalence of $\mathcal{M}$ and $\mathcal{N}$ is defined in terms of the \textit{factorisability} of a natural function associated with $\imath$. We therefore begin by introducing the notion of factorisability of general functions from $\mathscr{S}(G)$ to an abelian group $X$.

\begin{defi} (cf. \cite{Bartel14, cassou1989functions, deSmit})
Let $G$ be a finite group and $X$ be an abelian group written multiplicatively. A function $f \colon \mathscr{S}(G) \to X$ is said to be factorisable if there exists an injection of abelian groups $\psi : X \hookrightarrow Y$ for some $Y$ and a function $g \colon \mathrm{Irr}(G) \to Y$ defined on a full set $\mathrm{Irr}(G)$ of isomorphism classes of irreducible complex characters of $G$ such that
\begin{equation*}
    \psi(f(H)) = \underset{\chi \in \mathrm{Irr}(G)}{\prod} g(\chi)^{\langle \chi, \mathbb{C}[G/H] \rangle}
\end{equation*}
holds for all $H \in \mathscr{S}(G)$, where $\langle \chi, \mathbb{C}[G/H] \rangle$ denotes the multiplicity of $\chi$ in the representation $\mathbb{C}[G/H]$.
\end{defi}

In \cite{Burns1, Burns2, Frolich4}, Fr$\ddot{\textrm{o}}$hlich and Burns focused on the case when $G$ is abelian. In this situation, the factorisability of a function $f : \mathscr{S}(G) \longrightarrow X$ can be studied via its \textit{factorisable quotient} $\tilde{f}$. For a finite abelian group $G$, define a binary relation on $G$ by declaring $x,y \in G$ to be related if and only if they generate the same cyclic subgroup of $G$. This defines an equivalence relation, and we call each equivalence class $D$ a \textit{division} of $G$. Given a function $f : \mathscr{S}(G) \longrightarrow X$, we associate a function $f'$ on the set of divisions of $G$ with values in $X$ by
\begin{equation*}
    f'(D) : = \underset{C < \overline{D}}{\prod}f(C)^{\mu((\overline{D}:C))},
\end{equation*}
where $\overline{D}$ denotes the subgroup of $G$ generated by any element $x \in D$, and $\mu$ is the M$\ddot{\textrm{o}}$bius function. We define the factorisable quotient $\tilde{f} : \mathscr{S}(G) \longrightarrow X$ of $f$ by
\begin{equation*}
    \tilde{f}(H) : = \bigg ( \underset{ D \subset H}{\prod}f'(D) \bigg ) \cdot f(H)^{-1}
\end{equation*}
for every $H \in \mathscr{S}(G)$.
It is known that for general $f$, one always has $\tilde{f}(H)=1$ for all cyclic subgroups $H$ of $G$. We have the following proposition.
\begin{prop}\label{factorequivalenceabelian} (\cite{Frolich1, Frolich4})
Let $G$ be a finite abelian group and $f \colon \mathscr{S}(G) \to X$ a function from $\mathscr{S}(G)$ to an abelian group $X$. Then $f$ is factorisable if and only if we have $\tilde{f}(H)=1$ for all subgroups $H$ of $G$.
\end{prop}

With the notion of factorisability of a function, we can define the factor equivalence between two $\mathcal{O}_{K}[G]$-lattices.

\begin{defi}
Let $\mathcal{M}, \mathcal{N}$ and $\imath: \mathcal{M} \to \mathcal{N}$ be as in the beginning of this subsection. Two lattices $\mathcal{M}$ and $\mathcal{N}$ are said to be factor equivalent if the function from $\mathscr{S}(G)$ to the group $\mathrm{Id}_{K}$ defined by
\begin{equation*}
    H \longrightarrow [\, \mathcal{N}^{H} : \imath(\mathcal{M}^{H}) \, ]_{\mathcal{O}_{K}}
\end{equation*}
is factorisable, where $[\, \mathcal{N}^{H} : \imath(\mathcal{M}^{H}) \,]_{\mathcal{O}_{K}} \in \mathrm{Id}_{K}$  denotes the order ideal (cf. \cite[\S 80]{curtis1966representation}) of the $\mathcal{O}_{K}$-torsion module $\mathcal{N}^{H}/\imath(\mathcal{M}^{H})$.
\end{defi}

\begin{rema}(cf. \cite[Prop. 2.5]{deSmit})\label{rema-factor}
We record the following basic facts:
\begin{enumerate}
    \item[(i)] The definition of factor equivalence does not depend on the choice of $\imath$.
    \item[(ii)] The factor equivalence is an equivalence relation.
\end{enumerate}
\end{rema}

The following fact is well-known, but we provide a proof for the readers' convenience.

\begin{lemm}\label{genusimpliesfactor}
Let $\mathcal{M}$ and $\mathcal{N}$ be two $\mathcal{O}_K[G]$-lattices. If they are genus equivalent, then they are factor equivalent.
\end{lemm}

\begin{proof}
There is a canonical isomorphism $\mathrm{Id}_K \simeq \bigoplus_{\mathfrak{p}} \,\mathrm{Id}_{K_{\mathfrak{p}}}$ where $\mathfrak{p}$ runs over the maximal ideals of $\mathcal{O}_K$. Under this isomorphism, the ideal $[\, \mathcal{N}^H : \imath(\mathcal{M}^H) \,]_{\mathcal{O}_K}$ corresponds to the element
\begin{equation*}
\bigg ( \,\, [ \, (\mathcal{O}_{K_{\mathfrak{p}}} \otimes_{\mathcal{O}_K} \mathcal{N})^H : (1 \otimes \imath) \, (( \mathcal{O}_{K_{\mathfrak{p}}} \otimes_{\mathcal{O}_K} \mathcal{M})^H) \, ]_{\mathcal{O}_{K_{\mathfrak{p}}}} \,\, \bigg )_{\mathfrak{p}} \in \underset{\mathfrak{p}}{\bigoplus} \, \mathrm{Id}_{K_{\mathfrak{p}}}.
\end{equation*}
Hence, $\mathcal{N}$ and $\mathcal{M}$ are factor equivalent if and only if $\mathcal{O}_{K_{\mathfrak{p}}} \otimes_{\mathcal{O}_K} \mathcal{N}$ and $\mathcal{O}_{K_{\mathfrak{p}}} \otimes_{\mathcal{O}_K} \mathcal{M}$ are factor equivalent for all maximal ideals $\mathfrak{p}$. The latter can be checked by applying Remark \ref{rema-factor} (i) to the isomorphic $\mathcal{O}_{K_{\mathfrak{p}}}[G]$-lattices $\mathcal{O}_{K_{\mathfrak{p}}} \otimes_{\mathcal{O}_K} \mathcal{M}$ and $\mathcal{O}_{K_{\mathfrak{p}}} \otimes_{\mathcal{O}_K} \mathcal{N}$.
\end{proof}

Applying the theory of factor equivalence to arithmetic Galois modules has proved fruitful, with one lattice taken to be an arithmetic object and the other a standard (module-theoretic) lattice, as illustrated by the following examples (cf. \cite{deSmit, Frolich4}).

\begin{exam}\label{examfactor}
\begin{enumerate}
    \item[(i)] For a Galois extension $L/K$ of number fields, the normal basis theorem gives an isomorphism
    \begin{equation*}
     K \otimes_{\mathcal{O}_K} \mathcal{O}_L \simeq K[G_{L/K}] \simeq K \otimes_{\mathcal{O}_K} \mathcal{O}_K[G_{L/K}]   
    \end{equation*}
of $K[G_{L/K}]$-modules. Hence it is natural to study the factor equivalence between $\mathcal{O}_{L}$ and $\mathcal{O}_{K}[G_{L/K}]$.
    
    \item[(ii)] Let $L/K$ be a Galois extension of number fields, and let $S$ be a finite set of places of $L$ containing the set $S_{L, \infty}$ of all the archimedean places of $L$. Assume that $S$ is invariant under the action of $G_{L/K}$. Let $X_{S}$ denote the free abelian group generated by $S$, and let $Y_{S}$ be the kernel of the augmentation map $X_{S} \to \mathbb{Z}$, which maps each element of $S$ to $1$. By the generalized Dirichlet-Herbrand theorem on $S$-units (cf. \cite[Thm. I.3.7]{Gras3}), the multiplicative group $\mathcal{O}^{\times}_{L,S}$ of $S$-units of $L$ satisfies an isomorphism $$\mathbb{Q} \otimes_{\mathbb{Z}} \mathcal{O}^{\times}_{L,S} \simeq \mathbb{Q} \otimes_{\mathbb{Z}} Y_{S}$$ of $\mathbb{Q}[G_{L/K}]$-modules. Therefore, we can study the factor equivalence between $E_{L,S}:= \mathcal{O}^{\times}_{L,S}/\mu(L)$ and $Y_{S}$. We remark that if $L$ is a Galois extension over an admissible field $k$ where no infinite places of $k$ are ramified, then $Y_{S_{L,\infty}}$ is isomorphic to the augmentation ideal $I_{G_{L/k}}$ as $\mathbb{Z}[G_{L/k}]$-lattices.
\end{enumerate}
\end{exam}

\subsection{Theorems of Burns}\label{section Theorems of Burns}

As in \S \ref{Sec3.1}, let $K$ be a finite extension of $\mathbb{Q}$ or $\mathbb{Q}_p$. When $G$ is a finite abelian group, Burns \cite{Burns1} investigated when the factor equivalence of two $\mathcal{O}_K[G]$-lattices implies the genus equivalence. Building on this, in \cite{Burns2}, he obtained a necessary and sufficient condition for the genus equivalence of $E_{L}$ and $\mathcal{A}_{G_{L/k}}$, valid for all abelian extensions $L/k$ of admissible number fields $k$ unramified at the infinite places. In this subsection, we briefly explain these results, assuming throughout that $G$ is abelian and that $L/k$ is an abelian extension of an admissible field unramified at the infinite places.

\subsubsection{Arithmetic Criteria for the Existence of Local Minkowski Units in Abelian Extensions}

Let $\mathcal{A}$ be a $K$-algebra that is a quotient of $K[G]$, and let $X$ be an $\mathcal{O}_K[G]$-lattice such that $\mathcal{A}$ acts on $K \otimes_{\mathcal{O}_K} X$. We then set
\begin{equation*}
   \mathfrak{A}(\mathcal{A},X):= \{ \, \lambda \in \mathcal{A} \, | \, \lambda \cdot X \subseteq X \, \}.
\end{equation*}
%Let $\mathcal{A}$ be a $K$-algebra which is a quotient of $K[G]$. Let $X$ be an $\mathcal{O}_{K}[G]$-lattice such that $\mathcal{A}$ acts on $K \otimes_{\mathcal{O}_K} X$. Then, we define $\mathfrak{A}(\mathcal{A},X)$ to be the set $ \{ \, \lambda \in \mathcal{A} \, | \, \lambda \cdot X \subseteq X \, \}$.

Suppose now that $K \otimes_{\mathcal{O}_K} X$ is a subrepresentation of the regular representation $K[G]$. Then it is necessarily cyclic as a $K[G]$-module by the semisimplicity of $K[G]$. Since $G$ is abelian, the action of $K[G]$ factors through the unique quotient $K$-algebra $K(X)$ of $K[G]$. The induced action of $K(X)$ on $K \otimes_{\mathcal{O}_K} X$ is faithful. Consequently, $\mathfrak{A}(K(X),X)$ is an $\mathcal{O}_K$-order in $K(X)$, called the \textit{associated order} of $X$ in $K(X)$.

%it is isomorphic to a unique ideal $K(X)$ of $K[G]$, which is itself a $K$-algebra. In this situation, the role of $\mathcal{A}$ above is played by $K(X)$. The induced $K(X)$-module structure on $K \otimes_{\mathcal{O}_K} X$ is faithful, and hence $\mathfrak{A}(K(X),X)$ is an $\mathcal{O}_K$-order of $K(X)$. We call $\mathfrak{A}(K(X),X)$ the associated order of $X$ in $K(X)$.

%In particular, if the representation $K \otimes_{\mathcal{O}_K} X$ is a sub-representation of the regular representation $K[G]$, then by the commutativity of $G$, $K \otimes_{\mathcal{O}_K}X$ is isomorphic to a unique ideal $K(X)$ of $K[G]$, which is a $K$-algebra by itself. Since the complement of $K(X)$ in $K[G]$ annihilates $X$, the $K[G]$-module structure on $K \otimes_{\mathcal{O}_K} X$ induces a $K(X)$-module structure on $K \otimes_{\mathcal{O}_{K}} X$, which is faithful. By the faithfulness of the action, $\mathfrak{A}(K(X),X)$ is an $\mathcal{O}_K$-order of $K(X)$. We call $\mathfrak{A}(K(X),X)$ the \textit{associated order} of $X$ in $K(X)$. 

A normal subgroup $H$ of $G$ is called cocyclic (written $H <_{c} G$) if the quotient $G/H$ is cyclic. Burns introduced another equivalence relation on $\mathcal{O}_K[G]$-lattices called the \textit{order-equivalence}.

\begin{defi}(cf. \cite[\S 2]{Burns1}, \cite[\S 1]{Burns2})
Two $\mathcal{O}_{K}[G]$-lattices $X$ and $Y$ are said to be order-equivalent (written $X \circ Y$) if for every cocyclic subgroup $H$ of $G$ one has
\begin{equation*}
    \mathfrak{A}(K[G/H],X^{H}) = \mathfrak{A}(K[G/H],Y^{H}).
\end{equation*}
\end{defi}

Let $\mathfrak{M}_{K,G}$ denote the maximal $\mathcal{O}_{K}$-order in $K[G]$, and for every $\mathcal{O}_{K}[G]$-lattice $M$, let $M^{\mathfrak{M}_{K,G}}$ denote the maximal $\mathfrak{M}_{K,G}$-module contained in $M$. We also write $\hat{G}$ for the group of characters of $G$. For each subgroup $H$ of $G$, we let $H^{\perp}$ denote the subgroup of $\hat{G}$ consisting of characters that are trivial on $H$.

Given $\mathcal{O}_K[G]$-lattices $X$ and $Y$ with $K \otimes_{\mathcal{O}_K} X \simeq K \otimes_{\mathcal{O}_K} Y$ as $K[G]$-modules, we define the \textit{defect function} (cf. \cite[p. 260]{Burns1}, \cite[(1.16)]{Frolich1}) $$J(X,Y) : \mathscr{S}(\hat{G}) \to \mathrm{Id}_{K}.$$
For every subgroup $H^{\perp}$ of $\hat{G}$, it is given by 
\begin{equation*}
    J(X,Y)(H^{\perp}) = \frac{[X^{H} : (X^{\mathfrak{M}_{K,G}})^{H}]_{\mathcal{O}_{K}}}{[Y^{H} : (Y^{\mathfrak{M}_{K,G}})^{H}]_{\mathcal{O}_{K}}}.
\end{equation*}

The defect function is important in the works of Fr{\"o}hlich and Burns, since $X$ and $Y$ are factor equivalent if and only if $J(X,Y)$ is factorisable (cf. \cite[(1.17) in p. 411]{Frolich1}). Using the order-equivalence, Burns proved the following theorem. 

\begin{theo}(cf. \cite[Thm. 2]{Burns1})\label{factorimpliesgenus}
Let $K$ be a field of one of the following types:
\begin{enumerate}
    \item[$(i)$] a number field in which no prime divisor of $|G|$ ramifies in $K/\mathbb{Q}$, or
    \item[$(ii)$] an absolutely unramified local field, i.e. a finite unramified extension of $\Q_p$.
\end{enumerate}
Let $X$ be an $\mathcal{O}_{K}[G]$-lattice such that $K \otimes_{\mathcal{O}_K} X $ is isomorphic to a quotient $Q$ of $K[G]$, and let $\mathfrak{A}=\mathfrak{A}(Q,X)$ be the associated order of $X$ in $Q$. Then, $X$ and $\mathfrak{A}$ are genus equivalent if and only if both $X \circ \mathfrak{A}$ and $\widetilde{J(X,\mathfrak{A})}(\hat{G})=\mathcal{O}_{K}$ hold, where $\widetilde{J(X,\mathfrak{A})}$ denotes the factorisable quotient of the defect function $J(X,\mathfrak{A})$.
\end{theo}

\begin{rema}
The original formulation of Theorem \ref{factorimpliesgenus} in \cite[Thm.~2]{Burns1} is stated with $G$-$\circ$-equivalence in place of order equivalence. Since order equivalence implies $G$-$\circ$-equivalence while genus equivalence implies order equivalence, the present formulation, as used also in \cite{Burns2}, follows directly from Theorem 2 of \cite{Burns1}.
\end{rema}

We now return to the setting of local Minkowski units. Recall that $L/k$ is an abelian extension of an admissible field unramified at the infinite places. One easily checks that $\Q \otimes_\Z E_L$ is a subrepresentation of $\Q[G_{L/k}]$, and that $\mathbb{Q}(E_L)=A_{G_{L/k}}$ (the specialization of $K(X)$ with $K=\mathbb{Q},X=E_L$).

At first sight, Theorem \ref{factorimpliesgenus} appears to relate $E_{L}$ to $\mathfrak{A}(A_{G_{L/k}},E_{L})$, which contains $\mathcal{A}_{G_{L/k}}$. However, its significance lies in showing that certain arithmetic necessary conditions for the genus equivalence of $E_{L}$ and $\mathcal{A}_{G_{L/k}}$ are actually sufficient.

\vskip 5pt

The arithmetic necessary conditions are expressed in terms of the factorisable quotients of two functions $$h_{L/k},w_{L/k} : \mathscr{S}(\widehat{G_{L/k}}) \longrightarrow \mathbb{N},$$ where $\widehat{G_{L/k}}$ denotes the character group of $G_{L/k}$. They are defined by
\begin{equation*}
    h_{L/k}(H^{\perp}) : = \mathrm{lcm}( \, h_{L^{H}}, |G_{L/k}| \,), \qquad w_{L/k}(H^{\perp}) : = \mathrm{lcm}( \, w_{L^{H}}, |G_{L/k}| \,)
\end{equation*}
for every $H^{\perp} \in \mathscr{S}(\widehat{G_{L/k}})$. Here, $\mathrm{lcm}(a,b)$ denotes the least common multiple of $a, b \in \mathbb{N}$. For each abelian group $G$, define $$\tilde{\mathfrak{J}}_{G} := \big ( \underset{p}{\prod} \, p^{J_{p}} \big ) \cdot |G|^{-1},$$ where, for every prime $p$, we write $J_{p}$ for the number of non-trivial divisions of the $p$-Sylow subgroup of $G$. 

\begin{theo}(\cite[Thm. 3]{Burns2}) \label{Burns's theorem for genus equi}
Let $k$ be an admissible field. Let $L/k$ be an abelian extension unramified at the infinite places. Then, $E_{L}$ is genus equivalent to $\mathcal{A}_{G_{L/k}}$ if and only if we have both $$\tilde{h}_{L/k}(\widehat{G_{L/k}})=\tilde{w}_{L/k}(\widehat{G_{L/k}}) \cdot \tilde{\mathfrak{J}}_{G_{L/k}},$$ and $\hat{H}^{0}(H,E_{L})=1$ for every cocyclic subgroup $H$ of $G_{L/k}$.
\end{theo}

A noteworthy feature of this theorem is that it applies to all abelian Galois groups, since its proof does not rely on the classification of integral representations of $G_{L/k}$ over $\Z_p$. The functions $h_{L/k}$ and $w_{L/k}$ are related to the factor equivalence of $E_{L}$ and the lattice $Y_{S_{L,\infty}}$ introduced in Example \ref{examfactor} (ii) (cf. \cite[Thm. 5.2]{deSmit}, \cite[Thm. 7 (Multiplicative)]{Frolich4}). The invariant $\tilde{\mathfrak{J}}_{G}$ appears when considering the factor equivalence of $\mathcal{A}_{G}$ and $I_{G}$ (cf. \cite[page 75]{Burns2}, \cite{Frolich1}). Since $Y_{S_{L, \infty}}$ is isomorphic to $I_{G_{L/k}}$, this accounts for the appearance of these quantities in Theorem~\ref{Burns's theorem for genus equi}.

\vskip 5pt

\subsubsection{Applications of the Arithmetic Criteria to Genus Field Extensions}

The existence of local Minkowski units in a general Galois extension cannot be settled by representation-theoretic considerations alone. Since the arithmetic of general extensions is highly intricate, it is necessary to apply the arithmetic criteria in special cases where the number fields enjoy suitable arithmetic properties. Using Theorem~\ref{Burns's theorem for genus equi} together with the theory of central class fields (cf.~\cite{Frolich3}), Burns proved the existence of local Minkowski units for certain abelian $p$-extensions of admissible fields. This subsection briefly reviews these results and explains how they motivate our approach.

Throughout this subsection, let $k$ be an admissible field and $p$ a prime that does not divide $h_{k}w_{k}$. Following \cite{Burns2}, an abelian $p$-extension $L$ of $k$ is called a \textit{$p$-power genus field extension} if we have
\begin{equation*}
        G_{L/k} = \underset{v \in \Ram(L/k)}{\bigoplus} \, I_{L/k,v},
\end{equation*}
where $I_{L/k,v}$ denotes the inertia subgroup of $G_{L/k}$ at $v$. By a formula of Furuta \cite{Furuta}, $L$ is a $p$-power genus field extension of $k$ if and only if $p$ does not divide the genus number of $L$ over $k$.

\vskip 5pt

To state Burns's theorem, we recall some notation from \cite{UllomWatt}. For each finite place $v$ of $k$, let $\mathfrak{p}_v$ denote the maximal ideal of the valuation ring $\mathcal{O}_{k_v}$ for the local field $k_v$, and write $\mathbf{N}\mathfrak{p}_v \in \mathbb{N}$ for its ideal norm. Let $h_v$ be the smallest positive integer such that the $h_v$-th power of the prime ideal of $k$ corresponding to $v$ is principal, and fix a generator $\pi_v$ of this principal ideal.

If $v$ does not divide $p$, fix an element $x_v \in \mathcal{O}^{\times}_{k_v}$ whose class in the residue field $\kappa_v := \mathcal{O}_{k_v}/\mathfrak{p}_v$ generates the multiplicative group $\kappa_v^{\times}$. If $v$ divides $p$ and $I_{L/k,v}$ is cyclic, then fix $x_v \in \mathcal{O}^{\times}_{k_v}$ whose class in $$\mathcal{O}_{k_v}^{\times}/N_{L_w/k_v}\mathcal{O}_{L_w}^{\times} \simeq I_{L/k,v}$$ generates the group, where $w$ is a fixed prime of $L$ above $v$. For $x \in \mathcal{O}^{\times}_{k_v}$, define 

\begin{equation*}
[\, v,x \,] = \begin{cases}
m \bmod{(\mathbf{N}\mathfrak{p}_v -1)} \quad \text{if}\; v \nmid p \,\,\,\text{and}\,\, x \equiv x_v^m \pmod{\mathfrak{p}_v} \\
s \bmod{|I_{L/k,v}|} \,\,\,\,\,\qquad \text{if}\; v \mid p, \; \; I_{L/k,v} \; \text{is cyclic},\,\, \text{and}\,\, x \equiv x_v^s \pmod{N_{L_w/k_v}\mathcal{O}_{L_w}^{\times}}
\end{cases}
\end{equation*}

For finite places $v, v' \in \Ram(L/k)$, we consider $\pi_{v'} \in k^{\times}$ as an element of $\mathcal{O}_{k_v}$ and evaluate $[v, \pi_{v'}]$. Although $x_v$ and $\pi_v$ are chosen arbitrarily, this choice does not affect the $p$-divisibility of $[v,\pi_{v'}]$.

\vskip 5pt

In \cite{Frolich3, UllomWatt}, for an admissible field $k$ and a prime $p \nmid h_k w_k$, the $p$-power genus field extensions $L$ of $k$ with $p \nmid h_{L}$ were completely characterized in terms of the set $\Ram(L/k)$ and the $p$-divisibility of $[v_{i}, \pi_{v_{j}}]$ for $v_{i},v_{j} \in \Ram(L/k)$ (cf. \cite[Thm. 4]{Burns2}). Building on this, Burns \cite{Burns2} gave a complete characterization of the existence of local Minkowski units in such $L$.

As a preliminary remark, note that if $L$ is a $p$-power degree genus field extension of $k$ with $p \nmid h_{L}$ and $\mathrm{rk}_{p}(G_{L/k}) \leq 2$, then we have $|\Ram(L/k)|=\mathrm{rk}_{p}(G_{L/k})$ (cf. \cite[Thm. 4.(b)]{Burns2}). In particular, the group $I_{L/k,v}$ is cyclic for every $v \in \Ram(L/k)$. 

\begin{theo}(\cite[Thm. 5]{Burns2}) \label{genus equiv for genus extensions}
Let $k$ be an admissible field. Let $p$ be a prime not dividing $h_{k}w_{k}$. Let $L$ be a $p$-power degree genus field extension of $k$ with $p \nmid h_{L}$. Then $E_{L}$ and $\mathcal{A}_{G_{L/k}}$ are genus equivalent if and only if one of the following holds: 
\begin{enumerate}
    \item[(i)] $\Ram(L/k)=\{v_{1}\}$, and the group $G_{L/k}$ is cyclic;
    \item[(ii)] $\Ram(L/k)=\{v_{1},v_{2} \}$, the $p$-rank of $G_{L/k}$ is $2$, and both $[v_{1}, \pi_{v_{2}}]$ and $[v_{2},\pi_{v_{1}}]$ are not divisible by $p$.
\end{enumerate}
\end{theo}

\begin{rema}\label{criterionofBurns}
Consider the case $k=\mathbb{Q}$ with $p$ an odd prime. By class field theory, $I_{L/\mathbb{Q},v}$ is cyclic for every prime $v$ and every abelian $p$-extension $L/\mathbb{Q}$. Since $\mathbb{Q}$ has class number $1$, we may take $\pi_q=q$ for each rational prime $q$. It is known (cf. \cite[page 86]{Burns2}) that:
\begin{itemize}
\item[(i)] $[ \, p,q \,]$ is divisible by $p$ if and only if $q \equiv 1 \pmod{p^2}$, and
\item[(ii)] $[\, q,p \,]$ is divisible by $p$ if and only if $p$ is a $p$-th power residue modulo $q$.
\end{itemize}
\end{rema}
Let $L$ be a $p$-power genus field extension of $\mathbb{Q}$ ramified precisely at $p$ and $q$. If $[ \, p,q \,]$ is not divisible by $p$, then $h_L$ is prime to $p$ by a theorem of Fr$\ddot{\textrm{o}}$hlich (cf. \cite[Thm. 4.(b)]{Burns2}). In this case, $L$ is moreover $p$-rational (cf. Corollary \ref{totallyrealprational}).

In \S 6 we shall extend Theorem \ref{genus equiv for genus extensions} to \textit{non-abelian} $p$-extensions of $\mathbb{Q}$ unramified outside $p$ and $q$ such that $[\, p,q \,]$ is not divisible by $p$.

\subsection{Factor equivalence and regulator constants}\label{sectionfactorequivalence}

In this subsection, we present basic properties of the regulator constant that will be useful in later sections. We also recall the theorem of Bartel on the relationship between the theory of factor equivalence and the theory of regulator constants.

\subsubsection{Basic facts on regulator constants}

Let $G$ be a finite group and $\mathcal{R}$ a principal ideal domain with field of fractions $\mathcal{K}$. Throughout this subsection, we assume that $\mathcal{K}$ has characteristic prime to $|G|$.
The regulator constant $\mathcal{C}_{\Theta}(\mathcal{M})$ is an element of $\mathcal{K}^{\times}/\mathcal{R}^{\times \, 2},$ defined for every pair $(\mathcal{M},\Theta)$ of a $G$-relation $\Theta$ and an $\mathcal{R}[G]$-lattice $\mathcal{M}$ such that $\mathcal{K} \otimes_{\mathcal{R}} \mathcal{M}$ is a self-dual representation of $G$ over $\mathcal{K}$.
The theory of regulator constant was introduced by Tim and Vladimir Dokchitser in \cite{dokchitser2010birch} and has played a central role in several subsequent works.

\begin{defi}\label{G-relation}
A formal sum $\Theta = \sum_{H \leq G}n_H H$ of subgroups $H$ of $G$ with coefficients $n_H \in \mathbb{Z}$ is called a $G$-relation if there is an isomorphism
\begin{equation*}
 \underset{\substack{H \leq G \\ n_H < 0}}{\bigoplus} \, \mathbb{Q}[G/H]^{-n_H} \simeq \underset{\substack{H \leq G \\ n_H > 0}}{\bigoplus} \, \mathbb{Q}[G/H]^{n_H} 
\end{equation*}
of $\mathbb{Q}[G]$-modules.
\end{defi}

The set of $G$-relations forms a subgroup of the free abelian group $\mathbb{Z}[\mathscr{S}(G)]$ over the set $\mathscr{S}(G)$ of subgroups of $G$. Its $\mathbb{Z}$-rank is known to equal the number of conjugacy classes of non-cyclic subgroups of $G$ (cf. \cite[§13.1, Thm. 30]{serre1977linear}).
\begin{exam}\label{exam-G-relation}
\begin{enumerate}
    \item[(i)] If $G$ is cyclic, then there are no non-trivial $G$-relations.
    \item[(ii)] If $G$ is isomorphic to $(\mathbb{Z}/p\mathbb{Z})^2$ for a prime $p$, then the group of $G$-relations is generated over $\mathbb{Z}$ by the $G$-relation $$1+p\cdot G - \underset{H}{\sum} H,$$ where $H$ runs over the subgroups of $G$ of order $p$, and $1$ denotes the trivial subgroup.
\end{enumerate}
\end{exam}

\begin{defi}(\cite[Rem. 2.27]{dokchitser2009regulator})
Let $\mathcal{R}, \mathcal{K}, G$, and $\mathcal{M}$ be as above. Let $\langle\cdot,\cdot\rangle : \mathcal{M}\times \mathcal{M}\longrightarrow \mathcal{L}$ be a $\mathcal{R}$-bilinear, $G$-invariant pairing that is non-degenerate, with values in some extension $\mathcal{L}$ of $\mathcal{K}$. Let $\Theta=\sum_{H \leq G}n_{H}H$
be a $G$-relation. The regulator constant $\mathcal{C}_{\Theta}(\mathcal{M})$ of $\mathcal{M}$ with respect to $\Theta$ is defined by
\begin{equation*}
    \mathcal{C}_{\Theta}(\mathcal{M})=\prod_{H\leq G}\det{\left(\frac{1}{|H|}\langle\cdot,\cdot\rangle \big|_{\mathcal{M}^H}\right)^{n_H}}\in\mathcal{L}^{\times}/\mathcal{R}^{\times 2},
\end{equation*}
where each determinant is taken with respect to any $\mathcal{R}$-basis of $\mathcal{M}^{H}$.
\end{defi}

\begin{rema}\label{independenceonpairing}
\begin{enumerate}
    \item[(i)] It is known that $\mathcal{C}_{\Theta}(\mathcal{M})$ is independent of the particular choice of pairing (cf. \cite[Thm. 2.17]{dokchitser2009regulator}). Since $\mathcal{K} \otimes_{\mathcal{R}}\mathcal{M}$ is self-dual, there exists a non-degenerate $G$-invariant $\mathcal{K}$-bilinear pairing on $\mathcal{K}\otimes_{\mathcal{R}}\mathcal{M}$ with values in $\mathcal{K}$. Therefore, $\mathcal{C}_{\Theta}(\mathcal{M})$ is in fact defined in $\mathcal{K}^{\times}/\mathcal{R}^{\times 2}.$ 
    \item[(ii)] When $\mathcal{R}$ is equal to $\mathbb{Z}$, the regulator constants $\mathcal{C}_{\Theta}(\mathcal{M})$ take values in $\mathbb{Q}^{\times}$ because $\mathcal{R}^{\times \, 2}$ is trivial.
    \item[(iii)] The rational representations of the form $$\underset{H \leq G}{\bigoplus} \mathbb{Q}[G/H]^{a_H}, \qquad (a_H \in \mathbb{N})$$ are called \textit{permutation representations}. It is known that permutation representations are self-dual (cf. \cite[\S 3]{Bartel12}). Therefore, the regulator constant can be used to study $\mathbb{Z}[G]$-lattices whose rational representations are isomorphic to $A_{G} \oplus \mathbb{Q}[G]^m$ for $m \geq 0$, where $A_G$ denotes the representation $\mathbb{Q}[G]/(s_G)$.
    \item[(iv)] The readers can also refer to \cite[\S 3]{Bartel12} and \cite[Lem. 4.3]{Burns3} for other conceptual formulations of the regulator constant.
\end{enumerate}

\end{rema}

The following lemma is immediate from the definition.

\begin{lemm}\label{directproduct}
Let $\Theta$ and $\Theta'$ be $G$-relations, and let $\mathcal{M}$ and $\mathcal{M}'$ be $\mathcal{R}[G]$-lattices whose rational representations are self-dual. Then we have
\begin{equation*}
         \mathcal{C}_{\Theta}(\mathcal{M}\oplus \mathcal{M}')   =   \mathcal{C}_{\Theta}(\mathcal{M}) \cdot \mathcal{C}_{\Theta}(\mathcal{M}'), \qquad 
         \mathcal{C}_{\Theta+\Theta'}(\mathcal{M})   =  \mathcal{C}_{\Theta}(\mathcal{M}) \cdot \mathcal{C}_{\Theta'}(\mathcal{M}).
\end{equation*}
\end{lemm}

The following properties of $G$-relations and the regulator constants will be useful.

\begin{lemm}(\cite[Exam. 2.30]{dokchitser2009regulator})\label{constant}
If $\sum_{H \leq G}n_{H}H$ is a $G$-relation, then we have $\sum_{H \leq G}n_{H}=0$.
\end{lemm}

\begin{lemm}\label{cyclic subgroups}(\cite[Lem. 2.46]{dokchitser2009regulator}, \cite[Rem. 3.2]{Torzewski})
If $H$ is a cyclic subgroup of $G$, then we have $$\mathcal{C}_{\Theta}(\mathcal{R}[G/H])= 1$$ for every $G$-relation $\Theta$.
\end{lemm}

For a finite group $G$, there are two natural ways to construct $G$-relations from those of its subgroups and quotient groups.

\begin{enumerate}
    \item[(i)] Let $H$ be a subgroup of $G$ and let $\Theta =  \sum_{H' \leq H}n_{H'}H'$ be an $H$-relation. Then, $\Theta$ is also a $G$-relation, which we denote by $\textrm{Ind}_{H}^{G}\Theta$. 
    \item[(ii)] Let $B$ be a normal subgroup of $G$ and set $G'=G/B$. Suppose $\Theta'= \sum_{B \leq H \leq G}a_{H/B}(H/B)$ is a $G'$-relation. Then, $$\underset{B \leq H \leq G}{\sum}a_{H/B}H$$ is a $G$-relation, called the \textit{inflation} of $\Theta$, and is denoted by $\Inf_{G'}^{G}\Theta'$.
\end{enumerate}

If $\mathcal{M}$ is an $\mathcal{R}[G']$-lattice, then $\mathcal{M}$ can be viewed as an $\mathcal{R}[G]$-lattice via the natural projection $G \to G'$. We denote this $\mathcal{R}[G]$-lattice by $\Inf_{G'}^{G}\mathcal{M}$. Similarly, if $\mathcal{N}$ is an $\mathcal{R}[G]$-lattice and $H \leq G$, we denote by $\Res_{H}^{G}\mathcal{N}$ the corresponding $\mathcal{R}[H]$-lattice obtained by restriction. Then, we have the following proposition.

\begin{prop}(\cite[Prop. 2.45]{dokchitser2009regulator})
The following statements hold:
\begin{itemize}
\item[(i)] Let $G$ be a finite group and $G'$ a quotient of $G$. For every $G'$-relation $\Theta'$ and every $\mathcal{R}[G']$-lattice $\mathcal{M}$ with self-dual rational representation, we have
\begin{equation*}
    \mathcal{C}_{\Inf_{G'}^{G}\Theta'}(\Inf_{G'}^{G}\mathcal{M}) = \mathcal{C}_{\Theta'}(\mathcal{M}).
\end{equation*}

\item[(ii)] Let $H$ be a subgroup of $G$. For every $H$-relation $\Theta$ and every $\mathcal{R}[G]$-lattice $\mathcal{N}$ with self-dual rational representation, we have
\begin{equation*}
    \mathcal{C}_{\Ind_{H}^{G}\Theta}(\mathcal{N}) = \mathcal{C}_{\Theta}(\Res_{H}^{G}\mathcal{N}).
\end{equation*}

\end{itemize}
\end{prop}

\begin{proof}
Let $\langle \,, \, \rangle$ be a $G'$-invariant $\mathcal{R}$-bilinear non-degenerate pairing on $\mathcal{M}$. Then $\langle \, , \, \rangle$ is also a $G$-invariant non-degenerate pairing on $\Inf_{G'}^{G}\mathcal{M}$. The first equality follows from computing both regulator constants with $\langle \, , \, \rangle$. The second equality can be checked similarly by using a $G$-invariant $\mathcal{R}$-bilinear pairing on $\mathcal{N}$ as an $H$-bilinear pairing on $\Res_{H}^{G}\mathcal{N}$.
\end{proof}

Lastly, we mention a result on $v_p(\mathcal{C}_{\Theta}(\mathcal{M}))$ for rational primes $p$ when $\mathcal{R}$ is equal to $\mathbb{Z}$.

\begin{prop}(\cite[Prop. 3.9]{Bartel12}\label{pdivisibility})
Suppose that $\mathcal{R}$ is equal to $\mathbb{Z}$. Let $G$ be a finite group and $B$ be a normal subgroup of $G$ such that the quotient group $C=G/B$ is cyclic. Let $p$ be a prime not dividing $|B|$. Then, we have $v_p(\mathcal{C}_{\Theta}(\mathcal{M}))=0$ for every $G$-relation $\Theta$ and every $\mathbb{Z}[G]$-lattice $\mathcal{M}$ whose rational representation is self-dual.
\end{prop}

\subsubsection{Theorems of Bartel}

The regulator constant yields a new criterion to verify the factor equivalence of two $\mathbb{Z}[G]$-lattices.

\begin{theo}(\cite[Cor. 2.12]{Bartel14})\label{factor-regulatorconstant}
Let $G$ be a finite group. Let $\mathcal{M}$ and $\mathcal{N}$ be two $\mathbb{Z}[G]$-lattices with the same self-dual rational representation. Then, $\mathcal{M}$ and $\mathcal{N}$ are factor equivalent if and only if we have $$\mathcal{C}_{\Theta}(\mathcal{M})=\mathcal{C}_{\Theta}(\mathcal{N})$$ for all $G$-relations $\Theta$.
\end{theo}

The factor equivalence of two lattices can be studied locally, as shown in the following proposition.

\begin{prop}\label{factor-regulatorconstant ppart}
Let $G$ be a finite group and $p$ a prime. Let $\mathcal{M}$ and $\mathcal{N}$ be $\mathbb{Z}[G]$-lattices as in Theorem \ref{factor-regulatorconstant}. Then, $\mathbb{Z}_p \otimes_{\mathbb{Z}} \mathcal{M}$ and $\mathbb{Z}_p \otimes_{\mathbb{Z}} \mathcal{N}$ are factor equivalent as $\mathbb{Z}_{p}[G]$-lattices if and only if we have $$v_p(\mathcal{C}_{\Theta}(\mathcal{M})) = v_p(\mathcal{C}_{\Theta}(\mathcal{N}))$$ for all $G$-relations $\Theta$.
\end{prop}

\begin{proof}
For an injective $\mathbb{Z}[G]$-morphism $\imath \colon \mathcal{M} \to \mathcal{N}$ and a $G$-relation $\Theta = \sum_{H \leq G} n_H H$, we have
\begin{equation}\label{Bartel-relation}
\mathcal{C}_{\Theta}(\mathcal{M})/\mathcal{C}_{\Theta}(\mathcal{N}) = \underset{H \leq G}{\prod} \, [ \, \mathcal{N}^H : \imath(\mathcal{M}^H) \,]^{2n_H}
\end{equation}
(cf. \cite[Lem. 2.11]{Bartel14}). The claim follows from the equality
\begin{equation*}
[\, (\mathbb{Z}_p \otimes_{\mathbb{Z}} \mathcal{N})^{H} : (1 \otimes \imath )((\mathbb{Z}_p \otimes_{\mathbb{Z}} \mathcal{M})^{H}) \,]_{\mathbb{Z}_{p}} = |[\, \mathcal{N}^H : \imath(\mathcal{M}^H) \,]|_p
\end{equation*}
and \cite[Prop. 2.4.(4)]{Bartel14}.
\end{proof}

\begin{coro}\label{factor-p-group}
Let $G$ be a finite $p$-group, and let $\mathcal{M}$ and $\mathcal{N}$ be $\Z[G]$-lattices affording the same self-dual rational representation. 
Then $\mathcal{M}$ and $\mathcal{N}$ are factor equivalent if and only if we have
\[
    v_p(\mathcal{C}_{\Theta}(\mathcal{M})) = v_p(\mathcal{C}_{\Theta}(\mathcal{N}))
\]
for all $G$-relations $\Theta$.   
\end{coro}

\begin{proof}
By the proof of Lemma \ref{genusimpliesfactor}, $\mathcal{M}$ and $\mathcal{N}$ are factor equivalent if and only if $\mathbb{Z}_{\ell} \otimes_{\mathbb{Z}} \mathcal{M}$ and $\mathbb{Z}_{\ell} \otimes_{\mathbb{Z}} \mathcal{N}$ are factor equivalent as $\mathbb{Z}_{\ell}[G]$-lattices for every prime $\ell$. Since $G$ is a $p$-group, these lattices are isomorphic for all $\ell \neq p$. Hence, the claim follows from Proposition \ref{factor-regulatorconstant ppart}.

\end{proof}

In \cite{Frolich4}, Fr$\ddot{\textrm{o}}$hlich obtained a theorem \cite[Thm. 7 (Multiplicative)]{Frolich4} on the factor equivalence of $S$-units and $Y_{S}$ (cf. Example \ref{examfactor}). This theorem of Fr$\ddot{\textrm{o}}$hlich was generalized by de Smit, who also gave a simplified proof \cite[Thm. 5.2]{deSmit}. In \cite{Bartel12}, Bartel independently proved a theorem (Theorem \ref{Bartel}) on the regulator constant of the $S$-units. He later verified that his theorem is equivalent to the theorem of de Smit (cf. \cite[page 8]{Bartel14}) by using Theorem \ref{factor-regulatorconstant}. The following theorem of Bartel provides an arithmetic description of the regulator constant of the unit lattice, which is the principal integral lattice in this paper. Although his theorem is formulated for general $S$-units, we restrict here to the special case of the group of ordinary units.

\begin{theo}\label{Bartel}(\cite[Prop. 2.15]{Bartel12})
Let $L/K$ be a finite Galois extension of number fields with Galois group $G$. For each subgroup $H$ of $G$, write $\lambda(H)$ for the order of the kernel of the map
\begin{equation*}
H^{1}(H,\mu(L)) \longrightarrow H^{1}(H,\mathcal{O}_{L}^{\times})
\end{equation*}
induced by the inclusion $\mu(L) \hookrightarrow \mathcal{O}_{L}^{\times}$. If $\Theta = \sum_{H \leq G} n_H H$ is a $G$-relation, then we have
\begin{equation*}
\mathcal{C}_{\Theta}(E_L) = \mathcal{C}_{\Theta}(\mathbb{Z}) \cdot \underset{H \leq G}{\prod} \bigg ( \frac{R_{L^H}}{\lambda(H)} \bigg)^{2n_H},
\end{equation*}
where $\mathbb{Z}$ in $\mathcal{C}_{\Theta}(\mathbb{Z})$ denotes the $\mathbb{Z}[G]$-lattice $\mathbb{Z}$ with trivial $G$-action.
\end{theo}

\begin{rema}\label{regulatorconstantconstant}
Let $\Theta = \sum_{H \leq G}  n_H H$ be a $G$-relation. One checks  
\begin{equation*}
\mathcal{C}_{\Theta}(\mathbb{Z}) = \underset{H \leq G}{\prod} |H|^{-n_H}
\end{equation*}
using the bilinear map $\langle n,m \rangle := nm$ on $\mathbb{Z}$.
\end{rema}

\begin{rema}(\cite[Lem. 2.14]{Bartel12})\label{lambda}
Let $L/K$ be a Galois extension of number fields with Galois group $G$. Let $H$ be a subgroup of $G$. The embedding $\mathcal{O}_{L^{H}}^{\times} \hookrightarrow \mathcal{O}_{L}^{\times}$ induces an embedding $E_{L^{H}} \hookrightarrow E_{L}^{H}$. We can easily check
\begin{equation*}
    \lambda(H) = [E_{L}^{H} : E_{L^{H}}]
\end{equation*}
by using cohomology.
\end{rema}

\section{Regulator constants of some standard lattices}\label{standardlattices}

In this section, we derive formulae for $\mathcal{C}_{\Theta}(\mathcal{A}_G)$ and $\mathcal{C}_{\Theta}(I_G)$ that hold for general finite groups $G$ and $G$-relations $\Theta$. In \cite{Frolich1}, Fr{\"o}hlich proved that $\mathcal{A}_G$ and $I_G$ are not factor equivalent when $G$ is a non-cyclic abelian group. We provide a proof of this theorem by explicitly showing that if $G$ is not cyclic, then we have $\mathcal{C}_{\Theta}(\mathcal{A}_{G}) \neq \mathcal{C}_{\Theta}(I_G)$ for some $G$-relation $\Theta$. Our formulae will be useful in the study of the Galois module structure of unit lattices in later sections. When $G$ is cyclic, the lattices $\mathcal{A}_G$ and $I_G$ are factor equivalent, because $G$ has no non-trivial $G$-relations. Therefore, in this section, we assume that $G$ is non-cyclic.

\begin{prop}\label{Regulator constant C(Z[G]/G)}
Let $G$ be a non-cyclic finite group, and let $\Theta = \sum_{H \leq G} n_H H$ be a $G$-relation. Then we have
\begin{equation}\label{4.1.4}
\mathcal{C}_{\Theta}(\mathcal{A}_G) = \underset{H \leq G}{\prod}|H|^{n_H}.
\end{equation}
\end{prop}

\begin{proof}
By Remark \ref{independenceonpairing} (i), we may compute $\mathcal{C}_{\Theta}(\mathcal{A}_G)$ with any non-degenerate bilinear $G$-invariant pairing on $\mathcal{A}_G$. Let $(\, , \,)$ denote the pairing on $\mathbb{Z}[G]$ defined by $(g,g'):= \delta_{g,g'}$ for all $g,g' \in G$, where $\delta_{g,g'}$ denotes the Kronecker delta symbol. Define a second pairing $\langle \, , \, \rangle$ on $\mathbb{Z}[G]$ by
\begin{equation*}
\langle x, y \rangle := ( \, x - \frac{1}{|G|}\underset{g \in G}{\sum}gx \, , \, y - \frac{1}{|G|}\underset{g \in G}{\sum}gy \, )
\end{equation*}
for all $x,y \in \mathbb{Z}[G]$. A straightforward computation shows that $\langle g, g' \rangle=\delta_{g,g'}-1/|G|$ for all $g,g' \in G$. From this, one checks that both the left and right kernels of $\langle \, , \, \rangle$ coincide with the subspace $(s_G)$ of $\mathbb{Z}[G]$ generated by $s_G$. Hence, $\langle \, , \, \rangle$ induces a non-degenerate bilinear pairing on $\mathcal{A}_G$. The $G$-invariance of $\langle \, , \, \rangle$ on $\mathcal{A}_{G}$ follows directly from its $G$-invariance on $\mathbb{Z}[G]$.

 We can compute $\mathcal{C}_{\Theta}(\mathcal{A}_{G})$ by evaluating the determinant $\textrm{det} \big ( \langle , \rangle |_{(\mathcal{A}_{G})^{H}} \big )$ for each subgroup $H$ of $G$. For each subgroup $H$ of $G$, the set $\{ \overline{s_H \cdot \sigma} \}_{\sigma \not \in H}$ of the classes of $s_H \cdot \sigma \in \mathbb{Z}[G]$ with $\sigma \not \in H$, taken in $\mathcal{A}_G$, forms a $\mathbb{Z}$-basis of $\big ( \mathcal{A}_{G} \big )^{H}$.
By definition of $\langle \, , \, \rangle$, we have

\begin{equation*}
 \big \langle \,  \overline{s_H \cdot g_1}, \, \overline{s_H \cdot g_2} \, \big \rangle  = \begin{cases}
   \dfrac{|H|\cdot (|G|-|H|)}{|G|} \,\,\,\,\,& \text{if}\,\, Hg_1 = Hg_2, \\[8pt]
   -\dfrac{|H|^2}{|G|}\quad\quad\;\; \,\,\,\,&\text{if} \,\, Hg_1 \neq Hg_2.
   \end{cases} 
\end{equation*}

 Consequently, the matrix of $\big ( \frac{1}{|H|} \langle, \rangle|_{(\mathcal{A}_{G})^{H}} \big )$ with respect to the basis $\{ \overline{s_H \cdot \sigma} \}_{\sigma \not\in H}$ is the following circulant matrix of rank $(G:H)-1$:
\begin{equation}\label{circulant}
    \begin{pmatrix}
    \frac{|G|-|H|}{|G|} & -\frac{|H|}{|G|} & \cdots & -\frac{|H|}{|G|} & -\frac{|H|}{|G|} \\
   -\frac{|H|}{|G|} & \frac{|G|-|H|}{|G|} &  \cdots & -\frac{|H|}{|G|} & -\frac{|H|}{|G|} \\
    \vdots & \vdots & \ddots & \vdots & \vdots  \\
    -\frac{|H|}{|G|} & -\frac{|H|}{|G|} & \cdots & \frac{|G|-|H|}{|G|} & -\frac{|H|}{|G|} \\
    -\frac{|H|}{|G|} & -\frac{|H|}{|G|} & \cdots & -\frac{|H|}{|G|} & \frac{|G|-|H|}{|G|} 
    \end{pmatrix}.
\end{equation}
By the formula for the determinant of circulant matrices (cf. \cite[Thm. 1]{Conrad1}), the determinant of (\ref{circulant}) equals  
\begin{equation*}
   \frac{1}{|G|^{(G:H)-1}} \cdot \underset{j=0}{\overset{(G:H)-2}{\prod}} \, \bigg ( \, (|G|-|H|) - |H| \big ( \omega^{j}+\omega^{2j}+\cdots +\omega^{((G:H)-2)j} \big ) \bigg ),
\end{equation*}
where $\omega$ denotes a primitive $((G:H)-1)$th root of unity. The factor in the product equals $|H|$ when $j=0$, and $|G|$ otherwise. Therefore, the determinant $\textrm{det} \big ( \frac{1}{|H|}\langle , \rangle |_{(\mathcal{A}_{G})^{H}} \big )$ is equal to $|H||G|^{-1}.$ Hence, we have
\begin{equation}\label{4.1.7}
\mathcal{C}_{\Theta}(\mathcal{A}_{G}) = \underset{H \leq G}{\prod} \big ( |H|  |G|^{-1} \big )^{n_H}.
\end{equation}
By Lemma \ref{constant}, the exponent of $|G|$ in (\ref{4.1.7}) is $0$.
\end{proof}

\begin{prop}\label{augmentationideal}
Let $G$ be a non-cyclic finite group. Let $\Theta = \sum_{H \leq G} n_H H$ be a $G$-relation. Then we have
\begin{equation*}
    \mathcal{C}_{\Theta}(I_{G}) = \underset{H \leq G}{\prod}|H|^{-n_H}.
\end{equation*}
\end{prop}

\begin{proof}

We use the restriction to $I_{G}$ of the pairing $( \, , \, )$ defined on $\mathbb{Z}[G]$ by $( g , g' )=\delta_{g,g'}$ for $g,g'\in G$. This pairing is $G$-invariant and symmetric. Moreover, the left kernel of $( \, , \, )$ is trivial, because if $\sum_{\tau \in G}a_{\tau}\tau \in I_{G}$ lies in the left kernel, then we have
\begin{equation*}
   \big ( \, \underset{\tau \in G}{\sum}a_{\tau}\tau \, , \, \sigma -1 \, \big ) = a_{\sigma} - a_{1} = 0
\end{equation*}
for all $\sigma \in G$. From $\sum_{\tau \in G}a_{\tau}=0$, we deduce $a_{1}=0$. By the symmetry of the pairing, the right kernel is also trivial.

For every subgroup $H$ of $G$, the submodule $(I_{G})^{H}$ coincides with the kernel of the restriction of the augmentation map to $s_H \cdot \mathbb{Z}[G] = (\mathbb{Z}[G])^{H}$. Hence, $(I_{G})^{H}$ is equal to $s_H \cdot I_{G}$ with $\mathbb{Z}$-basis $\{ s_H \cdot (\sigma -1 ) \}_{H \sigma \neq H }$. For this basis, we have
\begin{equation*}
   \frac{1}{|H|} \cdot \big (\, s_H \cdot (\sigma-1) ,\, s_H \cdot (\tau -1 ) \,\big ) = \begin{cases}
   2 \,\,\,\,\,\text{if}\,\, H\sigma = H\tau, \\
   1 \,\,\,\,\,\text{otherwise.}
   \end{cases}  
\end{equation*}
Thus, the matrix $\bigg ( \, \frac{1}{|H|} \big (\,\, s_H \cdot (\sigma-1), \,\, s_H \cdot(\tau-1)\,\, \big  ) \, \bigg )_{H\tau, H\sigma \neq H} $ is the circulant matrix of rank $(G:H)-1$, whose diagonal entries are equal to $2$ and whose off-diagonal entries are equal to $1$. Its determinant is equal to 
\begin{equation}\label{augmentationproduct}
    \underset{j=0}{\overset{(G:H)-2}{\prod}} \,\bigg ( 2 + (\omega^{j}+\omega^{2j} + \cdots + \omega^{((G:H)-2)j} ) \bigg ),
\end{equation}
where $\omega$ denotes a primitive $((G:H)-1)$th root of unity. The product (\ref{augmentationproduct}) equals $(G:H)$.
In conclusion, we have 
\begin{equation*}
    \mathcal{C}_{\Theta}(I_{G}) = \underset{H \leq G}{\prod} (G : H)^{n_H} = \underset{H \leq G}{\prod} |H|^{-n_H}.
\end{equation*}
The last equality follows from Lemma \ref{constant}.

\end{proof}

The formulae for the regulator constants of $\mathcal{A}_{G}$ and $I_{G}$ yield the following corollary.

\begin{coro}\label{inverse}
For every non-cyclic group $G$ and every $G$-relation $\Theta$, we have $\mathcal{C}_{\Theta}(\mathcal{A}_{G}) = \mathcal{C}_{\Theta}(I_{G})^{-1}$. In particular, $\mathcal{A}_{G}$ and $I_{G}$ are not factor equivalent if $G$ is not cyclic.

\end{coro}

\begin{proof}
The first part of the statement follows immediately from Proposition \ref{Regulator constant C(Z[G]/G)} and Proposition \ref{augmentationideal}. By Theorem \ref{factor-regulatorconstant}, the $\mathbb{Z}[G]$-lattices $\mathcal{A}_{G}$ and $I_{G}$ are factor equivalent if and only if we have $\mathcal{C}_{\Theta}(\mathcal{A}_{G})=\mathcal{C}_{\Theta}(I_{G})$ for every $G$-relation $\Theta$. By Remark \ref{regulatorconstantconstant}, this is equivalent to $\mathcal{C}_{\Theta}(I_{G})=\mathcal{C}_{\Theta}(\mathbb{Z})=1$ for all $\Theta$. By \cite[Cor. 9.2]{Bartel-Dokchitser15}, for a prime $l$, there exists a $G$-relation $\Theta$ with $v_{l}(\mathcal{C}_{\Theta}(\mathbb{Z})) \neq 1$ if and only if $G$ has a subquotient isomorphic to $(\mathbb{Z}/l\mathbb{Z})^2$ or to $\mathbb{Z}/l\mathbb{Z} \rtimes \mathbb{Z}/p\mathbb{Z}$ for a prime $p$ such that $\mathbb{Z}/p\mathbb{Z}$ acts faithfully on $\mathbb{Z}/l\mathbb{Z}$. It remains to show that such a subquotient exists if $G$ is not cyclic.

Suppose that $G$ is a non-cyclic group. If $G$ has no subquotient of the form $(\mathbb{Z}/l\mathbb{Z})^2$ for any prime $l$, then all the Sylow subgroups of $G$ are cyclic. By \cite[Thm. 10.1.10]{robinson}, the group $G$ admits a presentation
\begin{equation*}
\langle \, a,b \, | \, a^m=1=b^n, \, b^{-1}ab=a^r \, \rangle
\end{equation*}
for some integers $m,n,r > 0$ with $(r-1)n$ relatively prime to $m$. For every quotient $\mathcal{Q}$ of $G$, write $a_{\mathcal{Q}}$ and $b_{\mathcal{Q}}$ for the classes of $a$ and $b$ in $\mathcal{Q}$ respectively. For any element $x$ of a group, let $\langle x \rangle$ be the cyclic subgroup generated by $x$. Since $r-1$ is coprime to $m$, the commutator subgroup $[G,G]$ is equal to $\langle a \rangle$. Hence, for every prime $l$ dividing $m$, the quotient $G(l):=G/\langle a^l \rangle$ is non-abelian. Let $G(l)'$ be the quotient of $G(l)$ by the subgroup of elements of $\langle b_{G(l)} \rangle$ that commute with $a_{G(l)}$. Then, the quotient $G(l)'$ is isomorphic to a semi-direct product $\mathbb{Z}/l\mathbb{Z} \rtimes \mathbb{Z}/n'\mathbb{Z}$ for some $n'|n$ such that $\mathbb{Z}/n'\mathbb{Z}$ acts faithfully on $\mathbb{Z}/l\mathbb{Z}$. For each prime divisor $p$ of $n'$, the subgroup of $G(l)'$ generated by $a_{G(l)'}$ and $b_{G(l)'}^{n'/p}$ is a non-abelian group of order $pl$.
\end{proof}

\section{Proof of Theorem \ref{maintheorem}}\label{Proof of the Main Theorem}

In this section, we prove Theorem \ref{maintheorem}. Let $L$ be a number field. From the analytic class number formula and the functional equation of the Dedekind zeta function $\zeta_L$, we obtain the following formula
\begin{equation}\label{analyticclassnumberformula}
    \zeta_{L}^{*}(0)=-\frac{h_{L}}{w_{L}}R_{L}
\end{equation}
for the special value of $\zeta_L(s)$ at $0$.
The Artin formalism for Artin $L$-functions, together with (\ref{analyticclassnumberformula}), yields the following theorem proved independently by Brauer \cite{Brauer} and by Kuroda \cite{Kuroda}.

\begin{theo}[Brauer--Kuroda]\label{Brauer-Kuroda}
Let $L/K$ be a Galois extension of number fields with Galois group $G$. If there is a $G$-relation $\Theta = \sum_{H \leq G} n_H H$, then we have the equality 
\begin{equation}\label{BrauerKuroda0}
\underset{H \leq G}{\prod} \bigg ( \frac{h_{L^H} R_{L^H}}{w_{L^H}} \bigg )^{n_H} = 1.
\end{equation}
\end{theo}

\noindent We shall use the following equivalent form of $(\ref{BrauerKuroda0})$, which is more convenient for our purpose :
\begin{equation}\label{BrauerKuroda}
\underset{H \leq G}{\prod} h_{L^H}^{n_H} = \big ( \underset{H \leq G}{\prod} R_{L^H}^{n_H} \big)^{-1} \times \big ( \underset{H \leq G}{\prod} w_{L^H}^{n_H} \big ).
\end{equation}

\begin{theo}[Theorem \ref{maintheorem}] \label{general necessary condition}
Let $L$ be a Galois extension of an admissible field $k$ with Galois group $G$. Suppose that no infinite places of $k$ are ramified in $L$. Then $E_L$ is factor equivalent to $\mathcal{A}_{G}$ as $\mathbb{Z}[G]$-lattices if and only if we have the equality
\begin{equation*}
\underset{H \leq G}{\prod} |H|^{n_H} =  \underset{ H \leq G}{\prod}h^{-n_H}_{L^H}  \cdot  \underset{H \leq G}{\prod} \left(\frac{\lambda(H)}{w_{L^H}}\right)^{-n_H} 
\end{equation*}
for all $G$-relations $\Theta = \sum_{H \leq G} n_H H$.
\end{theo}

\begin{proof}
By Theorem \ref{factor-regulatorconstant}, $E_{L}$ is factor-equivalent to $\mathcal{A}_{G}$ as $\mathbb{Z}[G]$-lattices if and only if we have $\mathcal{C}_{\Theta}(E_{L}) = \mathcal{C}_{\Theta}(\mathcal{A}_{G})$ for all $G$-relations $\Theta$. By Theorem \ref{Bartel} and Proposition \ref{Regulator constant C(Z[G]/G)}, this holds if and only if we have
\begin{equation}\label{5.1.1}
\underset{H \leq G}{\prod}|H|^{n_H} = \mathcal{C}_{\Theta}(\mathbb{Z}) \cdot \underset{H \leq G}{\prod} \left(\frac{R_{L^H}}{w_{L^H}}\right)^{2n_H}
\end{equation}
for all $G$-relations $\Theta = \sum_{H \leq G}n_H H$.
By the equality (\ref{BrauerKuroda}) and Remark \ref{regulatorconstantconstant}, the condition (\ref{5.1.1}) is equivalent to the condition
\begin{equation}\label{above}
 \underset{H \leq G}{\prod}|H|^{2n_H} = \underset{H \leq G}{\prod} h^{-2n_H}_{L^H} \cdot \underset{H \leq G}{\prod} \lambda(H)^{-2n_H} \cdot \underset{H \leq G}{\prod}w^{2n_H}_{L^H}.
\end{equation}
Taking the positive square root of \eqref{above} yields the claim.
\end{proof}

\begin{rema}(\cite[\S 2]{Brauer})\label{quotientsofw}
Brauer proved that for every Galois extension $L/K$ of number fields and every $G_{L/K}$-relation $\Theta = \sum_{H \leq G_{L/K}} n_H H$, the quotient $\prod_{H \leq G_{L/K}}w_{L^{H}}^{n_H}$ is a power of $2$.
\end{rema}

Let $L$ and $G$ be as in Theorem \ref{general necessary condition}. If $L$ is totally real, then we have $w_{L^{H}}=2$ for all subgroup $H$ of $G$. Then, the quotients of $w_{L^{H}}$'s are equal to $1$ by Lemma \ref{constant}. Furthermore, if $G$ has an odd order, then we have $\lambda(H)=1$ for every subgroup $H$ of $G$. Therefore, we have the following corollary.

\begin{coro}
Let $L$ be a real Galois number field of odd degree. Let $G$ be the Galois group of $L$ over $\mathbb{Q}$. Then, the $\mathbb{Z}[G]$-lattices $E_L$ and $\mathcal{A}_{G}$ are factor equivalent if and only if we have the equality  
\begin{equation*}
\underset{H \leq G}{\prod} |H|^{n_H} = \underset{H \leq G}{\prod}h^{-n_H}_{L^H}
\end{equation*}
for all $G$-relations $\Theta=\sum_{H \leq G}n_H H$.
\end{coro}

\begin{rema}
Theorem \ref{maintheorem} is a generalization of the necessary conditions on the quotient of class numbers of subfields for the existence of local Minkowski units that were obtained by Burns \cite[Thm. 3]{Burns2}, Duval \cite[Rem. 5.3 (a)]{duval1981structure}, and Marszalek \cite[Thm. 2.8. (b)]{Mars2}.
\end{rema}

\begin{exam}
Let $G = (\mathbb{Z}/p\mathbb{Z})^2 \rtimes \mathbb{Z}/p\mathbb{Z}$ be the Heisenberg group of order $p^3$ with $p \geq 3$.
It has the $G$-relation $$\Theta= I-IZ-J+JZ,$$ where $Z$ denotes the center of $G$, and $I$ and $J$ are two non-conjugate, non-central subgroups of order $p$. Let $L$ be a Galois extension of $\mathbb{Q}$ with Galois group $G$. Then, the factor equivalence of $E_{L}$ and $\mathcal{A}_{G}$ is subject to the following necessary condition
    \begin{equation*}
        \frac{h_{L^I}h_{L^{JZ}}}{h_{L^J}h_{L^{IZ}}}=1.
    \end{equation*}
\end{exam}

\begin{rema}
For a finite group $G$, a $G$-relation $\Theta = \sum_{H \leq G} n_H H$ is called \textit{useful} (cf. \cite{biasse2022norm}) if we have $n_{1} \neq 0$ for the trivial subgroup $1$ of $G$. For a certain useful $G$-relation $\Theta$ (cf. \cite[Assumption 1.3]{bartel2013index}), the index $$[\mathcal{M} : \sum_{n_H \cdot n_1 < 0}\mathcal{M}^H] \in \mathbb{N}$$ is finite for every $\mathbb{Z}[G]$-lattice $\mathcal{M}$. In \cite{bartel2013index}, the authors studied the relationship between this index, the rational representation $\mathbb{Q} \otimes_{\mathbb{Z}} \mathcal{M}$, and the factor equivalence class of $\mathcal{M}$ for $G$-relations satisfying \cite[Assumption 1.3]{bartel2013index}. As a consequence, if $\mathcal{M}$ is the unit lattice of a Galois extension $L$ of an admissible field $k$ where no infinite places of $k$ are ramified in $L$, then the index (for the $G$-relations satisfying \cite[Assumption 1.3]{bartel2013index}) is uniquely determined by the factor equivalence class of $\mathcal{M}$.
\end{rema}

\begin{rema}
In \cite{bartel2013index}, the authors showed that the regulator constant can be used to obtain analogous results on the integral Galois module structure of higher K-groups of number fields and the Mordell-Weil groups of elliptic curves.
\end{rema}

\section{The arithmetic properties of totally real $p$-rational number fields}\label{sec-p-rational}

The theory of factor equivalence provides a method to study the Galois module structure of unit lattices in terms of class numbers. However, applying the theory \textit{in practice} is usually difficult because class numbers are notoriously hard to compute. Burns exploited the strong arithmetic properties of \( p \)-power genus field extensions of admissible fields to study the existence of local Minkowski units (cf. \S\ref{section Theorems of Burns}). 

In the remainder of this paper, we examine the Galois module structure of unit lattices for another special family of totally real number fields, called $p$-rational fields. The $p$-rational fields were investigated in \cite{JaulentNguyen, Mova-Do, Mova2} to construct infinitely many non-abelian extensions of \( \mathbb{Q} \) satisfying Leopoldt’s conjecture at the prime \( p \). It has long been observed that many arithmetic problems become simpler when the field is $p$-rational. In the totally real case, this principle appears to be more amenable to direct treatment, since the Galois group of the maximal pro-$p$ extension unramified outside $p$ has a simpler structure (cf.~\cite[Figure~1]{GrasGreen}), which permits more straightforward methods of relating the defect of $p$-rationality to the complexity of the problems (cf.~\cite{GrasTS, GrasGreen}). Motivated by this perspective, we apply the strong arithmetic properties of $p$-rational fields to prove the non-existence of Minkowski units in non-abelian $p$-rational $p$-extensions of $\Q$ (\S\ref{Factor equivalence class of the unit groups of $p$-rational}), and to study the relative Galois module structure of unit lattices in Galois extensions of totally real $p$-rational fields (\S\ref{relGalmod}).

\vskip 5pt

The abundance of $p$-rational fields in our context is illustrated by the following two facts:
\begin{itemize}
    \item[(i)] By a theorem of Movahhedi (Theorem \ref{p-rational p-extensions}), if a number field \( F \) is \( p \)-rational, then there exists an infinite family of infinite pro-$p$ towers of \( p \)-rational \( p \)-extensions of \( F \).
    \item[(ii)] It is widely believed that a number field \( F \) is \( p \)-rational for many primes \( p \). In \cite{Gras2}, Gras even conjectured that \( F \) is \( p \)-rational for all but finitely many primes \( p \).
\end{itemize}
Thus, our results on the Galois module structure of unit lattices apply to a large family of number fields.

We will focus on the Galois module structure of unit lattices in non-cyclic Galois extensions of number fields, as the theory of factor equivalence becomes trivial when the Galois group is cyclic (cf. Proposition \ref{factorequivalenceabelian} and the remark preceding it).

\subsection{Totally real $p$-rational number fields}\label{Totally real $p$-rational number fields}

Let \( F \) be a number field and \( p \) an odd prime. We write \( F_{\infty} \) for the cyclotomic \( \mathbb{Z}_p \)-extension of \( F \). For each integer \( n \geq 0 \), let \( F_n \) denote the \( n \)th layer of the extension \( F_{\infty}/F \). We write \( H_F \) for the \( p \)-Hilbert class field of \( F \), and \( \mathfrak{h}_{F} \) for the \( p \)-class number of \( F \). For a set \( S \) of primes of \( F \), let \( F_S \) denote the maximal pro-\( p \) extension of \( F \) unramified outside \( S \), and write \( G_S(F) \) for the Galois group of \( F_S \) over \( F \). We denote by \( S_p \) the set of \( p \)-adic primes of \( F \). By local class field theory, a non-$p$-adic prime $\mathfrak{q}$ of $F$ can ramify in a pro-$p$ extension of $F$ only if its ideal norm $\mathbf{N}\mathfrak{q}$ is congruent to $1$ modulo $p$ (see \cite[\S 8.5]{Koch} for an elementary explanation). Thus, without this congruence, the situation is vacuous, and we may assume that every non-$p$-adic prime in $S$ satisfies this congruence.

In this subsection, we briefly recall some arithmetic properties of totally real \( p \)-rational number fields that will be useful later. Except for Proposition \ref{descending}, Conjecture \ref{Gras conjecture}, and Theorem \ref{p-rational p-extensions}, we assume throughout that \( F \) is totally real. For more general information on \( p \)-rational number fields, the reader is referred to \cite{Gras3, JaulentNguyen, Mova, Mova-Do, Mova2, BenMova}.

Several equivalent characterizations of $p$-rationality can be found in the literature, for example, in \cite[page 22]{Mova}. 

\begin{defi}
A number field \( F \) is said to be \( p \)-rational if one of the following equivalent conditions holds: 
\begin{enumerate}
    \item The Galois group \( G_{S_p}(F) \) is a free pro-\( p \) group,  where \( S_p \) is the set of \( p \)-adic primes of \( F \);
    \item We have an isomorphism $G_{S_p}(F)^{\ab} \simeq \mathbb{Z}_p^{c_F + 1}$, where \( c_F \) denotes the number of complex places of \( F \).
\end{enumerate}  
\end{defi}

In particular, a totally real number field \( F \) is \( p \)-rational precisely when we have
\[
G_{S_p}(F) \simeq G_{S_p}(F)^{\ab} \simeq \mathbb{Z}_p.
\]

\begin{exam}
By the Kronecker–Weber theorem, the rational number field $\mathbb{Q}$ is $p$-rational for every prime $p$.
\end{exam}

We now return to the setting of a general totally real $p$-rational field. Since $F_{\infty}$ is a subfield of $F_{S_p}$, we have \( F_{S_p} = F_{\infty} \) in this case. Therefore, $H_F$ is a subfield of $F_{\infty}$, and we obtain the following proposition.

\begin{prop}\label{pclassnumber}
Let $F$ be a totally real $p$-rational number field. Let $m$ be the largest integer such that $F_{m}/F$ is unramified. Then, we have $\mathfrak{h}_{F}=p^m$.
\end{prop}

As an immediate consequence of Proposition \ref{pclassnumber}, we can observe the following lemma. We also note in passing that, in fact, in a Galois $p$-extension of totally real $p$-rational number fields, at most one non-$p$-adic prime can be ramified, while there is no such restriction on the $p$-adic places (cf. Corollary \ref{totallyrealprational}).

\begin{lemm}\label{p-class number cyclic}
Let $L/F$ be an extension of totally real $p$-rational number fields. Then, the following claims are valid.
\begin{enumerate}
\item[(i)] We have the inequalities $\mathfrak{h}_{F} \cdot |[L:F]|_p^{-1} \leq \mathfrak{h}_{L} \leq \mathfrak{h}_{F} \cdot [L:F]$.
\item[(ii)] If $L/F$ is a cyclic extension of degree $p$ that is unramified outside $p$, then we have $\mathfrak{h}_{L} = \mathrm{max}\{\mathfrak{h}_{F}/p, 1\}$.
\item[(iii)] If $L/F$ is a cyclic extension of degree $p$ that is ramified precisely at a non-$p$-adic prime, then we have $\mathfrak{h}_{L}=\mathfrak{h}_{F}$.
\end{enumerate}
\end{lemm}

\begin{proof}
\begin{enumerate}
\item[(i)] The left inequality follows from the inclusion $LH_F \subseteq H_L$ and the equalities $$[LH_F:L] = [H_F: H_F \cap L] = \mathfrak{h}_F \cdot [H_F \cap L : F]^{-1}.$$ Since $H_L$ is contained in $L_{\infty}=L F_{\infty}$, there is some $m \in \mathbb{N}$ such that $H_L$ is equal to the compositum of $L$ and $F_m$. Every $p$-adic prime of $F$ has ramification index at most $[L:F]$ in the extension $H_L/F$. Hence, the degree $[F_m:F]$, which is bounded above by the product of $\mathfrak{h}_F$ and the maximal ramification index of the $p$-adic primes of $F$ in the extension $F_m/F$, is in turn bounded above by $\mathfrak{h}_F \cdot [L:F]$. Therefore, we have $\mathfrak{h}_L=[H_L:L] \leq [F_m:F] \leq \mathfrak{h}_F \cdot [L:F]$.
\item[(ii)] If $L/F$ is unramified outside $p$, then $L$ is equal to $F_1$. If $F$ has $p$-class number $1$, then at least one prime of $F$ is totally ramified in $F_{\infty}$. Hence, we have $\mathfrak{h}_F=\mathfrak{h}_L=1$. If we have $H_F \neq F$, then we have $H_F=H_L$ because $H_F/L$ is unramified and $F_{\infty}/H_F$ is totally ramified at some $p$-adic prime. Thus, we have $\mathfrak{h}_L=[H_F:L]=p^{-1}\mathfrak{h}_F$.
\item[(iii)] Let $r$ be an integer such that $H_F$ is equal to $F_{r-1}$. Then $F_r/F$ is ramified at some $p$-adic prime of $F$ say $\mathfrak{p}$. Since $\mathfrak{p}$ is unramified in $L/F$, $F_r L/L$ is ramified at the primes of $L$ above $\mathfrak{p}$. Thus, we have $H_L \subsetneq F_r L$ and consequently $H_L=H_F L$. The claim follows because $L$ and $H_F$ are linearly disjoint over $F$.
\end{enumerate}
\end{proof}

%Note that the field $\Q$ is $p$-rational for every prime number $p$ and that an imaginary quadratic field $F$ is $p$-rational 
%for all but finitely many primes (cf. \cite[Proposition 4.1]{Greenberg}). We recall a conjecture of Gras about the $p$-rationality of number fields. 

The $p$-rationality of number fields satisfies the following descending property.

\begin{prop}(cf. \cite[Thm. I.1]{Gras6}, \cite[Prop. 5 on page 30]{Mova})\label{descending}
Let $L/F$ be an extension of number fields. If $L$ is $p$-rational, then $F$ is also $p$-rational.
\end{prop}

%By Proposition \ref{pclassnumber} and Proposition \ref{descending}, we can study the $p$-part of quotients of class numbers of subfields in Theorem \ref{maintheorem} and investigate the factor equivalence class of $\mathbb{Z}_p \otimes_{\mathbb{Z}} E_{F}$ when $F$ is a totally real $p$-rational Galois extension of $\mathbb{Q}$. Therefore, studying the existence of Galois extensions of totally real $p$-rational fields with various Galois groups can also be interesting for studying the Galois module structure of unit lattices. For finite $p$-groups, we can use the results of Movahhedi \cite{Mova2} on the ascent of $p$-rationality in Galois $p$-extensions of number fields. The problem for general groups is not well understood. Recent results have been obtained for groups of the form $(\mathbb{Z}/2\mathbb{Z})^{t}$ for $t\geq 1$ (cf. \cite{chattopadhyay2023p, Greenberg,triquadratic,Maire2023}).

By Propositions \ref{pclassnumber} and \ref{descending}, we can study the \( p \)-part of the quotients of class numbers of subfields appearing in Theorem \ref{maintheorem}, and investigate the factor equivalence class of \( \mathbb{Z}_p \otimes_{\mathbb{Z}} E_F \) when \( F \) is a totally real \( p \)-rational Galois extension of \( \mathbb{Q} \). Therefore, understanding the existence of Galois extensions of totally real \( p \)-rational fields with various Galois groups is also of interest in the study of the Galois module structure of unit lattices.

For finite \( p \)-groups, one can apply the results of Movahhedi \cite{Mova2} on the ascent of \( p \)-rationality in Galois \( p \)-extensions of number fields. However, the problem for general groups remains poorly understood. Some recent progress has been made for groups of the form \( (\mathbb{Z}/2\mathbb{Z})^t \) for \( t \geq 1 \) (cf.~\cite{chattopadhyay2023p, Greenberg, triquadratic, Maire2023}).

We also record the following conjecture due to Gras, which indicates that studying $p$-rational fields can yield results of a rather general nature for number fields.

\begin{conjecture}\label{Gras conjecture}(\cite{Gras2})
A number field is $p$-rational for all but finitely many primes.
\end{conjecture}

\vskip 5pt

To state the theorem of Movahhedi (Theorem \ref{p-rational p-extensions}), we first recall the notion of a primitive set of places (cf. \cite{Gras7, Mova}). Let $F$ be a number field. Let $F_{S_p}(1)$ denote the maximal elementary abelian extension of $F$ in $F_{S_p}$. Hence, if $F$ is $p$-rational, then the Galois group $\mathrm{Gal}(F_{S_p}(1)/F)$ of $F_{S_p}(1)$ over $F$ is isomorphic to $ (\mathbb{Z}/p\mathbb{Z})^{c_{F}+1}$ as a vector space over the finite field $\mathbb{Z}/p\mathbb{Z}$.

%\begin{rema}\label{tameprime}
%By the local class field theory, a non-$p$-adic prime $\mathfrak{q}$ of $F$ can be ramified in a pro-$p$ extension only if its norm $\mathrm{N}\mathfrak{q}$ is congruent to $1$ modulo $p$ (for an elementary explanation, we refer the readers to \cite[\S 8.5]{Koch}). In this section, we always assume that the ramified non-$p$-adic primes satisfy this congruence.
%\end{rema}

\begin{defi}
Let $F$ be a number field. Let $S$ be a finite set of finite places of $F$ containing $S_{p}$. The set $S$ is called \textit{primitive} for $(F,p)$ if the set of Frobenius automorphisms in $\mathrm{Gal}(F_{S_p}(1)/F)$ at the finite non-$p$-adic primes of $S$ are linearly independent over $\mathbb{Z}/p\mathbb{Z}$.
\end{defi}

With this notion, we can now state the following theorem.

\begin{theo}(\cite[Thm.2]{Mova2})\label{p-rational p-extensions}
Let $F$ be a $p$-rational number field. Let $L$ be a Galois $p$-extension of $F$. Then, $L$ is $p$-rational if and only if the set $\Ram(L/F) \cup S_{p}$ is primitive for $(F,p)$, where $\Ram(L/F)$ denotes the set of primes ramified in $L/F$.
\end{theo}

\begin{rema}
The readers can also refer to \cite{Gras6,Gras5} for a class field theoretic approach on the ascent of $p$-rationality under the Leopoldt conjecture at $p$.
\end{rema}

%As a consequence of Theorem \ref{p-rational p-extensions}, we obtain the following criterion in the totally real case.

\begin{coro}\label{totallyrealprational}
Let $F$ be a totally real $p$-rational number field. For a non-$p$-adic prime $\mathfrak{q}$ of $F$, the set $S_p \cup \{\mathfrak{q}\}$ is primitive for $(F,p)$ if and only if $\mathfrak{q}$ does not split in $F_1$. It follows that if $L$ is a Galois $p$-extension of $F$, then $L$ is $p$-rational if and only if there exists such a prime $\mathfrak{q}$ with $\Ram(L/F) \subseteq S_p \cup \{\mathfrak{q}\}.$
\end{coro}

\begin{rema}
If $F$ is $p$-rational, then it is easily seen that $p$-rationality ascends in the pro-$p$ tower $F_{S_p}/F$, since every closed subgroup of a free pro-$p$ group is free (cf. \cite[Cor. 3 on page 31]{Serre-GC}). In comparison, Theorem \ref{p-rational p-extensions} addresses the ascent in larger pro-$p$ towers.
\end{rema}

%By Chebotarev's density theorem, Theorem \ref{p-rational p-extensions} gives us infinitely many towers of $p$-rational $p$-extensions of a number field $F$ as soon as $F$ is known to be $p$-rational.\\ 
%The group structure of $G_{S}(F)$ is well understood by a theorem of Movahhedi when $F$ is $p$-rational and $S$ is primitive for $(F,p)$. 
%We use this to find non-abelian $p$-rational fields that we consider to study the existence of Minkowski units.
%We state his theorem for the case when $F$ is totally real.

By Chebotarev's density theorem, Theorem \ref{p-rational p-extensions} ensures the existence of infinitely many towers of \( p \)-rational \( p \)-extensions of a number field \( F \), provided that \( F \) itself is \( p \)-rational.

\vskip 3pt

Moreover, when \( F \) is \( p \)-rational and the set \( S \) is primitive for the pair \( (F, p) \), the structure of the Galois group \( G_S(F) \) is well understood, thanks to a theorem of Movahhedi. We state Movahhedi’s theorem below in the case where \( F \) is totally real.

\begin{prop}(\cite[Thm. 3.3]{Mova-Do})\label{Demushkin}
Let $F$ be a totally real $p$-rational number field. Let $\mathfrak{q}$ be a non-$p$-adic prime of $F$ such that $S=S_{p} \cup \{\mathfrak{q}\}$ is primitive for $(F,p)$. Then, $G_{S}(F)$ is the Demu\v{s}kin group of rank $2$ with minimal presentation
\begin{equation*}
    \langle \,\, \sigma, \tau \,\, | \,\, \tau^{\mathbf{N}\mathfrak{q}-1}[\tau,\sigma]=1 \,\,\rangle,
\end{equation*}
where $\mathbf{N}\mathfrak{q}$ is the ideal norm of the prime ideal $\mathfrak{q}$.
\end{prop}

%Proposition \ref{Demushkin} gives us information on the $G$-relations necessary for applying Proposition \ref{factor-regulatorconstant ppart} to $\mathbb{Z}_p \otimes_{\mathbb{Z}} E_{K}$ for a Galois $p$-extension $K/F$ contained in $F_{S}$. In particular, the following fact is important in this work.

By Proposition \ref{Demushkin}, the group $G_S(F)$ is a Demu\v{s}kin group of rank $2$. This yields information on the $G$-relations needed to apply Proposition \ref{factor-regulatorconstant ppart} to the $\mathbb{Z}_p[G_{L/F}]$-lattice $\mathbb{Z}_p \otimes_{\mathbb{Z}} E_L$, where $L/F$ is a Galois $p$-extension contained in $F_S$ (cf. \S \ref{Factor equivalence class of the unit groups of $p$-rational}). Moreover, the group structure can be exploited to obtain more precise information on the $p$-class group (cf. \S \ref{uniquepprime}).

\begin{prop}(\cite[Lem. 2.5]{Mova-Do})\label{elementaryabelianS}
Let $F$ be a totally real $p$-rational number field. Let $S$ be a finite set of primes of $F$ containing $S_{p}$. Then we have
\begin{equation*}
    G_{S}(F)^{\ab} \simeq \mathbb{Z}_{p} \times \underset{v}{\prod}\, \mathbb{Z}/\mu_{p}(F_{v})\mathbb{Z},
\end{equation*}
where $v$ runs over the finite non-$p$-adic primes of $S$, and $\mu_{p}(F_{v})$ denotes the number of $p$-power roots of unity in the completion $F_v$ of $F$ at $v$.
\end{prop}

In particular, Proposition \ref{elementaryabelianS} implies the following corollary, which may also be understood from the fact that every open subgroup of a Demu\v{s}kin group of rank $2$ has generator rank $2$.

\begin{coro}\label{elementaryabelian}
Let $\mathfrak{q}$ be a finite non-$p$-adic prime of $F$ that does not split in $F_{\infty}$. Then, the maximal elementary abelian extension of $F$ in $F_{S_p \cup \{\mathfrak{q}\}}$ is the compositum of $F_{1}$ and a cyclic $p$-extension of $F$ in which $\mathfrak{q}$ is ramified (note that $p$ may also ramify in that cyclic extension).
\end{coro}

By the Burnside Basis theorem, Corollary \ref{elementaryabelian} implies the following.

\begin{lemm}(\cite[Thm. 2]{Mova2})\label{nonsplitting}
Let $F$ be a totally real $p$-rational number field and $\mathfrak{q}$ a non-$p$-adic prime such that $S_{p} \cup \{\mathfrak{q}\}$ is primitive for $(F,p)$. Then, $\mathfrak{q}$ does not split in $F_{S_{p} \cup \{\mathfrak{q}\}}$. In this situation, for any finite extension $L$ of $F$ contained in $F_{S_p \cup \{\mathfrak{q}\}}$, we denote by $\mathfrak{q}_L$ the unique prime of $L$ above $\mathfrak{q}$.
\end{lemm}

For a totally real $p$-rational field $F$ and a finite non-$p$-adic prime $\mathfrak{q}$ of $F$ such that 
$S_p \cup \{\mathfrak{q}\}$ is primitive for $(F,p)$, we will frequently consider finite extensions 
$L/F$ contained in the tower $F_{S_p \cup \{\mathfrak{q}\}}$. 
By Corollary~\ref{elementaryabelian}, for each such extension $L$, there exists a unique elementary 
abelian extension of $L$ in $F_{S_p \cup \{\mathfrak{q}\}}$. 
Since these extensions play a central role, we fix the following notation.

For the base field $F$, we denote by $F^{\el}_{p, \mathfrak{q}}$ the maximal elementary abelian extension 
of $F$ contained in $F_{S_p \cup \{\mathfrak{q}\}}$. By definition and Lemma~\ref{nonsplitting}, we have
\[
F_{S_p \cup \{\mathfrak{q}\}} = L_{S_p \cup \{\mathfrak{q}_L\}} 
\quad \text{and} \quad
\mathrm{Gal}(F_{S_p \cup \{\mathfrak{q}\}}/L) = G_{S_p \cup \{\mathfrak{q}_L\}}(L).
\]
Hence, we define $L^{\el}_{p, \mathfrak{q}_L}$ analogously, and, for simplicity, write $L^{\el}_{p, \mathfrak{q}}$ instead. 
By Corollary~\ref{elementaryabelian}, this is the unique extension of $L$ in $F_{S_p \cup \{\mathfrak{q}\}}$ with $G_{L_{p, \mathfrak{q}}^{\el}/L} \simeq (\Z/p\Z)^2.$

Since the base field $F$ and the prime $\mathfrak{q}$ are clear from the context, the shorthand $L_{p, \mathfrak{q}}^{\el}$ will cause no ambiguity. In particular, we obtain the following field diagram.

%For a totally real $p$-rational number field $F$ and a non-$p$-adic prime $\mathfrak{q}$ of $F$ such that $S_{p} \cup \{\mathfrak{q}\}$ is primitive for $(F,p)$, we will frequently analyze a finite extension $N$ of $F$ in $F_{S_p \cup \{\mathfrak{q}\}}$ and the maximal elementary abelian extension of $N$ in $F_{S_{p} \cup \{\mathfrak{q}\}}$.
%In what follows, we use the notation $N_{p, \mathfrak{q}}(1)$ to denote the latter number field. By Corollary \ref{elementaryabelian}, $N_{p,\mathfrak{q}}(1)$ is the unique extension of $N$ in $F_{S_{p} \cup \{\mathfrak{q}\}}$ such that we have $G_{N_{p,\mathfrak{q}}(1)/N} \simeq (\mathbb{Z}/p\mathbb{Z})^{2}$. In particular, we have the following diagram

\begin{equation*}
\begin{tikzcd}[row sep=large]
  & & F_{S_p \cup \{\mathfrak{q}\}} = L_{S_p \cup \{\mathfrak{q}_L\}} \\
  & & \\
  & & {L^{\mathrm{el}}_{p,\mathfrak{q}_L}:=L^{\el}_{p, \mathfrak{q}}} 
      \arrow[uu, no head] 
      \arrow[uu, "\Phi(G_{S_p \cup \{\mathfrak{q}_L\}}(L))"', no head, bend right=49] \\
  & L 
      \arrow[ru, no head] 
      \arrow[ruuu, "G_{S_p \cup \{\mathfrak{q}_L\}}(L)" description, no head, bend left] 
      \arrow[ru, "(\mathbb{Z}/p\mathbb{Z})^2"', no head, bend right=49] & \\
  F 
      \arrow[ru, no head] 
      \arrow[rruuuu, "G_{S_p \cup \{\mathfrak{q}\}}(F)" description, no head, bend left=60] & & 
\end{tikzcd}
\end{equation*}
where $\Phi(G_{S_p \cup \{\mathfrak{q}_L\}}(L))$ denotes the Frattini subgroup of $G_{S_p \cup \{\mathfrak{q}_L\}}(L)$.

%\begin{rema}\label{nonsplitting}
%Let $\mathfrak{q}$ be a finite non-$p$-adic prime of $F$ such that $S_{p} \cup \{\mathfrak{q}\}$ is primitive for $(F,p)$. Let $L$ be an extension of $F$ in $F_{S_{p} \cup \{\mathfrak{q}\}}$. Then, there is a unique prime $\mathfrak{q}'$ of $L$ above $\mathfrak{q}$, and we have $L_{S_{p} \cup \{\mathfrak{q}'\}}=F_{S_{p} \cup \{\mathfrak{q}\}}$. In this work, we will write $\mathfrak{q}_L$ for the unique prime of $L$ above $\mathfrak{q}$. Also, we will use the notation $L_{p,\mathfrak{q}}(1)$ to mean $L_{p,\mathfrak{q}_{L}}(1)$ by abuse of notation.
%\end{rema}

\subsection{On the $p$-class numbers in a pro-$p$ tower of totally real $p$-rational number fields when $p$ does not split}\label{uniquepprime}

We have seen in Proposition \ref{pclassnumber} and Lemma \ref{p-class number cyclic} that the $p$-class numbers of totally real $p$-rational fields are relatively easy to analyze. When there is a unique $p$-adic prime, these results can be refined further. This refinement will play an important role in the proof of Theorem \ref{maintheorem2} in the next section. In order to make this refinement precise, we now consider a totally real $p$-rational field $F$ together with a finite non-$p$-adic prime $\mathfrak{q}$, and introduce the following hypothesis on the triple $(F, p, \mathfrak{q})$.

%Let $F$ be a totally real $p$-rational field and $\mathfrak{q}$ be a finite non-$p$-adic prime satisfying the following hypothesis:
\begin{equation*}\label{condU}
\tag{U}
S_p \cup \{\mathfrak{q}\} \text{ is primitive for } (F,p),
\text{and } p \text{ does not split in } F_{S_p \cup \{ \mathfrak{q}\}}.
\end{equation*}
%By Burnside basis theorem, $(F,p, \mathfrak{q})$ satisfies (U) if and only if the local degree of the $p$-adic prime of $F$ at $F_{p,\mathfrak{q}}(1)$ is equal to $p^2$. In this subsection, we obtain a structural result on the $p$-class numbers of number fields in the tower $F_{S_p \cup \{\mathfrak{q}\}}/F$ for $(F,p, \mathfrak{q})$ satisfying (U) by considering the inertia subgroup of $G_{S_p \cup \{\mathfrak{q}\}}(F)$ at the unique $p$-adic prime of $F_{S_{p} \cup \{\mathfrak{q}\}}$. For each number field $F \subseteq L \subset F_{S_p \cup \{\mathfrak{q}\}}$, we will write $p_L$ for the unique prime of $L$ above $p$. The inertia subgroup of $G_{S_p \cup \{\mathfrak{q}\}}(F)$ at the $p$-adic prime corresponds to the maximal tamely ramified extension $F_{\{\mathfrak{q}\}}$ of $F$ in $F_{S_p \cup \{\mathfrak{q}\}}$.

By the Burnside basis theorem, the triple \((F, p, \mathfrak{q})\) satisfies condition \eqref{condU} precisely when $F$ has a unique $p$-adic prime and the local degree of this prime in the extension \(F_{p,\mathfrak{q}}^{\el}\) is \(p^2\). In this subsection, we obtain a structural result on the \(p\)-class numbers of number fields in the tower \(F_{S_p \cup \{\mathfrak{q}\}}/F\) for those \((F, p, \mathfrak{q})\) satisfying \eqref{condU}, by exploiting the inertia subgroup of \(G_{S_p \cup \{\mathfrak{q}\}}(F)\) at the unique \(p\)-adic place of \(F_{S_p \cup \{\mathfrak{q}\}}\). For each number field \(F \subseteq L \subset F_{S_p \cup \{\mathfrak{q}\}}\), we denote by \(p_L\) the unique prime of \(L\) lying above \(p\).

The inertia subgroup of \(G_{S_p \cup \{\mathfrak{q}\}}(F)\) at the unique $p$-adic place of $F_{S_p \cup \{\mathfrak{q}\}}$ corresponds to the maximal tamely ramified extension \(F_{\{\mathfrak{q}\}}\) of \(F\) contained in \(F_{S_p \cup \{\mathfrak{q}\}}\).

\begin{lemm} (cf. \cite[Thm. 1.1]{LeeLim})
Let $F$ be a totally real $p$-rational number field. Let $\mathfrak{q}$ be a non-$p$-adic prime of $F$ such that $S_p \cup \{\mathfrak{q}\}$ is primitive for $(F,p)$. Then, the extension $F_{\{\mathfrak{q}\}}/F$ is finite.
\end{lemm}

\begin{proof}
Since $G_{\{\mathfrak{q}\}}(F)$ is a quotient of the Demu\v{s}kin group $G_{S_p \cup \{\mathfrak{q}\}}(F)$, it is powerful (cf.~\cite[Chap.~3]{analytic}) with generator rank at most $2$. By \cite[Thm.~8.32]{analytic}, $G_{\{\mathfrak{q}\}}(F)$ admits an open uniformly powerful subgroup $\mathcal{U}$. By \cite[Thm.~3.8]{analytic}, the generator rank of $\mathcal{U}$ is at most $2$. 
If $G_{\{\mathfrak{q}\}}(F)$ is infinite, then so is $\mathcal{U}$, and in this case $\mathcal{U}$ admits a quotient isomorphic to $\Z_p$ (cf.~\cite[Exercise~3.11]{analytic}). It follows that if $F_{\{\mathfrak{q}\}}/F$ were infinite, then $F_{\{\mathfrak{q}\}}$ would contain a $\Z_p$-extension of a finite extension of $F$, which is impossible since in any $\Z_p$-extension of a number field at least one $p$-adic prime must ramify. Therefore $F_{\{\mathfrak{q}\}}/F$ is finite.
\end{proof}

\begin{prop}\label{pclassnumber-inertiasubgroup}
Suppose that $(F,p, \mathfrak{q})$ satisfies \eqref{condU}. Then, we have $\mathfrak{h}_L=1$ for every number field $L$ with $F_{\{\mathfrak{q}\}} \subseteq L \subset F_{S_p \cup \{\mathfrak{q}\}}$.
\end{prop}

\begin{proof}
Let $L$ be a number field as above. Then $p_L$ is totally ramified in $F_{S_p \cup \{ \mathfrak{q}\}}$ because $\mathrm{Gal}(F_{S_p \cup \{\mathfrak{q}\}}/L)$ is a subgroup of the inertia subgroup of $G_{S_p \cup \{\mathfrak{q}\}}(F)$ at the $p$-adic place of $F_{S_p \cup \{\mathfrak{q}\}}$. Since $p_L$ is ramified in $L_1$, the conclusion follows from Proposition \ref{pclassnumber}.
\end{proof}

\begin{exam}\label{Upq}
Note that $(\mathbb{Q},p,q)$ satisfies \eqref{condU} if and only if both $[ \, p,q \,]$ and $[ \, q,p \,]$ are not divisible by $p$ (cf. Remark \ref{criterionofBurns}). For example, the triple $(\mathbb{Q}, 7, 71)$ satisfies \eqref{condU}. In that case, the $p$-class number is $1$ along the tower $\mathbb{Q}_{S_p \cup \{q\}}/\mathbb{Q}_{\{q\}}$ by Proposition \ref{pclassnumber-inertiasubgroup}.
\end{exam}

\begin{rema}
It may be difficult to generalize Proposition \ref{pclassnumber-inertiasubgroup} to a pro-$p$ tower $F_{S_p \cup \{\mathfrak{q}\}}/F$ in which $p$ splits, since one need to take into account all the inertia subgroups of $G_{S_p \cup \{\mathfrak{q}\}}(F)$ at the $p$-adic places. For example, one may recall that the $p$-class field tower of $F$ is the subfield of $F_{S_{p}}$ fixed by the inertia subgroups of $G_{S_p}(F)$ at the $p$-adic places.
\end{rema}

\begin{coro}\label{totally ramified}
Suppose that $(F,p, \mathfrak{q})$ satisfies \eqref{condU}. Let $L$ be a finite extension of $F$ in $F_{S_p \cup \{\mathfrak{q}\}}$. The extension $F_{S_p \cup \{\mathfrak{q}\}}/L$ is totally ramified at $p_L$ if and only if $F_{\{\mathfrak{q}\}}$ is a subfield of $L$.
\end{coro}

\begin{proof}
Suppose that $p_L$ is totally ramified in $F_{S_p \cup \{\mathfrak{q}\}}/L$. Then we have $F_{\{\mathfrak{q}\}}L=L$ because $F_{\{\mathfrak{q}\}}L/L$ is unramified at $p_L$. The sufficiency is trivial.
\end{proof}

Now, we study the $p$-class numbers in a cyclic extension $L/K$ of number fields with $[L:K]=p$ contained in $F_{S_p \cup \{\mathfrak{q}\}}/F$ for a triple $(F,p,\mathfrak{q})$ satisfying \eqref{condU}. By (ii) and (iii) of Lemma \ref{p-class number cyclic}, it remains to treat the case where $L/K$ is ramified both at $p_K$ and $\mathfrak{q}_K$.

\begin{prop}\label{cyclic pq ramified}
Suppose that $(F,p,\mathfrak{q})$ satisfies \eqref{condU}. Let $L/K$ be a cyclic extension of number fields of degree $p$ contained in $F_{S_p \cup \{\mathfrak{q}\}}/F$, ramified at both $p_K$ and $\mathfrak{q}_K$. Then, we have
\begin{equation*}
    \mathfrak{h}_L = \begin{cases} \mathfrak{h}_K & \text{if $F_{\{\mathfrak{q}\}} \subseteq H_K$,} \\
    p \cdot \mathfrak{h}_K & \text{otherwise.}
    \end{cases}
\end{equation*}
\end{prop}

\begin{proof}
For all integers $m \geq 0$, we have $L_m = L K_m$ because $L/K$ is ramified at $\mathfrak{q}_K$. By the $p$-rationality of $K$, there is some integer $n \geq 0$ such that we have $H_K=K_n$. Then, $\mathfrak{h}_L$ is equal to one of $p^n$ and $p^{n+1}$ by Lemma \ref{p-class number cyclic} (i).

By Proposition \ref{pclassnumber}, we have $\mathfrak{h}_L = p^{n+1}$ if and only if $L_{n+1}/L_{n}$ is unramified. Since $K_n/K$ is unramified, both $p_{K_n}$ and $\mathfrak{q}_{K_n}$ are ramified in $L_n/K_n$. Since the ramification index of $\mathfrak{q}_{K_n}$ in $L_{n+1}/K_n$ is $p$, we can check that $G_{L_{n+1}/K_{n}}$ is isomorphic to $(\mathbb{Z}/p\mathbb{Z})^{2}$. In particular, $L_{n+1}$ is equal to $(K_n)^{\el}_{p,\mathfrak{q}}$.

If $L_{n+1}/L_{n}$ is ramified at $p_{L_n}$, then $G_{L_{n+1}/K_{n}}$ coincides with the inertia subgroup at $p_{L_{n+1}}$. By the Burnside basis theorem, this occurs if and only if $p_{K_n}$ is totally ramified in $F_{S_p \cup \{\mathfrak{q}\}}$ (cf. \cite[Chap. II, Thm. 10.7]{Neukirch}). Hence, the claim follows from Corollary \ref{totally ramified}.
\end{proof}

Even though condition \eqref{condU} may appear restrictive, we can find many triples $(F,p,\mathfrak{q})$ satisfying \eqref{condU} under the Gras Conjecture (Conjecture \ref{Gras conjecture}). Let $F$ be a totally real cyclic extension of $\mathbb{Q}$. According to the conjecture, there are expected to be infinitely many rational primes
$p$ such that
\begin{itemize}
    \item $F$ is $p$-rational,
    \item $p\nmid [F:\mathbb{Q}]$,
    \item $p$ does not split in $F$.
\end{itemize}
In the following proposition, we assume that $p$ is such a prime, so that $F$ is $p$-rational. Examples of this type can be found in \cite{Lim22}.

\begin{prop}\label{existence 6.2}
Let $F$ and $p$ be as above. Let $q$ be a rational prime such that $(\mathbb{Q},p,q)$ satisfies \eqref{condU} (cf. Example \ref{Upq}). Then, for every prime $\mathfrak{q}$ of $F$ above $q$, the hypothesis \eqref{condU} is satisfied by $(F,p,\mathfrak{q})$.
\end{prop}
\begin{proof}

Let $S_{q}$ be the set of primes of $F$ above $q$. Since $p$ is prime to $[F:\mathbb{Q}]$, every element $\mathfrak{q}$ of $S_{q}$ does not split in $F_{\infty}=F\mathbb{Q}_{\infty}$ by the primitivity of the set $S_p \cup \{q\}$. Hence, the set $S_{p} \cup \{\mathfrak{q} \}$ is primitive for $(F,p)$ for every $\mathfrak{q} \in S_{q}$ by Corollary \ref{totallyrealprational}.

It remains to show that the unique $p$-adic prime $\mathfrak{p}$ of $F$ does not split in $F_{S_p \cup \{\mathfrak{q}\}}$ for every $\mathfrak{q} \in S_{q}$. By the Burnside basis theorem, this happens if and only if $\mathfrak{p}$ does not split in $F_{p, \mathfrak{q}}^{\el}$ (cf. \cite[Chap. II, Prop. 9.6]{Neukirch}). Since $[F:\mathbb{Q}]$ is prime to $p$, $\mathfrak{p}$ is totally ramified in $F_{\infty}/F$. In particular, we have $\mathfrak{h}_F=1$. Therefore, the ramification index of $\mathfrak{p}$ in $F_{p,\mathfrak{q}}^{\el}/F$ is at least $p$, and $\mathfrak{p}$ splits in $F_{p,\mathfrak{q}}^{\el}$ only if there exists a cyclic extension $F(\mathfrak{q})$ of $F$ of degree $p$ in which $\mathfrak{p}$ splits and $\mathfrak{q}$ is ramified. Such an extension $F(\mathfrak{q})$ is unique if it exists.

Since $G_{F/\mathbb{Q}}$ acts transitively on $S_{q}$, the pro-$p$ extensions $\{ F_{S_p \cup \{\mathfrak{q}\}} \}_{\mathfrak{q} \in S_{q}}$ are conjugate to each other over $\mathbb{Q}$. Hence, if $F(\mathfrak{q})$ exists for some $\mathfrak{q} \in S_{q}$, then $F(\mathfrak{q})$ exists for every $\mathfrak{q} \in S_{q}$. Let $\mathscr{F}$ be the compositum of the fields $F(\mathfrak{q})$ for all $\mathfrak{q} \in S_{q}$. Then, $\mathfrak{p}$ splits completely in $\mathscr{F}$.

On the other hand, by class field theory and the triviality of the $p$-class group of $F$, $\mathscr{F}$ contains the compositum of $F$ and the subfield $\mathscr{K}$ of $\mathbb{Q}(\zeta_q)$ with $[\mathscr{K}:\mathbb{Q}]=p$. Since $p$ does not split in $\mathbb{Q}_{S_p \cup \{q\}}$, the residue class degree of $p$ in $\mathscr{K}$ is $p$. Hence, the residue class degree of $\mathfrak{p}$ in $\mathscr{F}$ is divisible by $p$, a contradiction.
\end{proof}

\begin{rema}\label{infinitude pq}
In fact, for every odd prime $p$, there exist infinitely many primes $q$ such that $(\mathbb{Q},p,q)$ satisfy \eqref{condU}. For a proof, we refer the readers to the application of the Gras-Munnier theorem in \cite{BLM}.
\end{rema}

\section{Non-existence of Minkowski units in non-abelian $p$-rational\\ $p$-extensions of $\mathbb{Q}$}\label{Factor equivalence class of the unit groups of $p$-rational}

In this section, we will prove Theorem \ref{maintheorem2}. Let $F$ be a non-abelian $p$-rational Galois $p$-extension of $\mathbb{Q}$. Since $G_{F/\mathbb{Q}}$ is a $p$-group, we may apply the theorem of Tornehave and Bouc on the generators of $G$-relations of finite $p$-groups $G$. We shall show that the factor equivalence class of $E_{F}$ can be analyzed via Galois extensions $L/K$ of subfields of $F$ with $G_{L/K} \simeq (\mathbb{Z}/p\mathbb{Z})^2$. To begin, let us recall the theorem of Tornehave and Bouc.

\begin{theo}(cf. \cite[Thm. 5.3]{Bartel-Dokchitser15}, \cite[Cor. 6.16]{Bouc})\label{Bouc's theorem}
Let $G$ be a finite $p$-group. Then, all $G$-relations are $\mathbb{Z}$-linear combinations of ones of the form $\Ind_{H}^{G}\Inf_{H/B}^{H}\Theta$ for pairs $(H/B, \Theta)$ of subquotients $H/B$ of $G$ and $H/B$-relations $\Theta$ of the following types:
\begin{enumerate}
    \item[(i)] $H/B \simeq (\mathbb{Z}/p\mathbb{Z})^2$ with the $H/B$-relation $\Theta$
    \begin{equation*}
        1 - \underset{C}{\sum}\, C + p \cdot H/B,
    \end{equation*}
    where $C$ runs over all the subgroups of $H/B$ of order $p$.
    
    \item[(ii)] $H/B$ is the Heisenberg group of order $p^{3}$ and $\Theta$ is the $H/B$-relation 
    \begin{equation*}
        I - IZ - J + JZ,
    \end{equation*}
    where $Z$ is the center of $H/B$ and $I,J$ are two non-conjugate non-central subgroups of $H/B$ of order $p$.
    \item[(iii)] $H/B$ is isomorphic to the dihedral group $D_{2^{n}}$ for some $n \geq 4$ and $\Theta$ is the $H/B$-relation 
        \begin{equation*}
        I - IZ - J + JZ,
    \end{equation*}
    where $Z$ is the center of $H/B$ and $I,J$ are two non-conjugate non-central subgroups of $H/B$ of order $2$.
\end{enumerate}
\end{theo}
\vskip 10pt

Let $L/K$ be a Galois $p$-extension of number fields. Let $H/B$ be a subquotient of $G_{L/K}$, and let $$\Theta = \underset{B \leq H' \leq H}{\sum} n_{H'}(H'/B)$$ be an $H/B$-relation. A straightforward computation shows that we have $$\textrm{Ind}_{H}^{G_{L/K}}\textrm{Inf}_{H/B}^{H}\Theta = \underset{B \leq H' \leq H}{\sum} n_{H'}H'.$$ By Theorem \ref{Bartel}, we then obtain the equality

\begin{equation}\label{Bouctheorem1}
    \mathcal{C}_{\textrm{Ind}_{H}^{G_{L/K}}\textrm{Inf}_{H/B}^{H}\Theta}(E_{L}) = \mathcal{C}_{\textrm{Ind}_{H}^{G_{L/K}}\textrm{Inf}_{H/B}^{H}\Theta}(\mathbb{Z}) \cdot \underset{B \leq H' \leq H}{\prod} \, \bigg (\frac{R_{L^{H'}}}{\lambda(H')} \bigg)^{2n_{H'}}.
\end{equation}
Applying Theorem \ref{Brauer-Kuroda}, (\ref{Bouctheorem1}) becomes 
\begin{equation}\label{Bouctheorem2}
    \mathcal{C}_{\textrm{Ind}_{H}^{G_{L/K}}\textrm{Inf}_{H/B}^{H}\Theta}(E_{L}) = \mathcal{C}_{\textrm{Ind}_{H}^{G_{L/K}}\textrm{Inf}_{H/B}^{H}\Theta}(\mathbb{Z}) \cdot \underset{B \leq H' \leq H}{\prod} \bigg ( \frac{w_{L^{H'}}}{\lambda(H') \cdot h_{L^{H'}}} \bigg )^{2n_{H'}}.
\end{equation}
Since $G_{L/K}$ is a $p$-group, Proposition \ref{pdivisibility} implies that we have
\begin{equation*}
v_l \big ( \,\mathcal{C}_{\textrm{Ind}_{H}^{G_{L/K}}\textrm{Inf}_{H/B}^{H}\Theta}(E_{L} \,) \big)=0 \quad  \text{for all primes $l \neq p$.} 
\end{equation*}
Since the same holds for the regulator constants of the lattice $\mathbb{Z}$, the product in \eqref{Bouctheorem2} involving $w_{L^{H'}}$, $\lambda(H')$, and $h_{L^{H'}}$ 
can be replaced by its $p$-part. In particular, if $L$ is totally real, then we obtain the following observation.

\begin{prop}\label{pextensions}
Let $L/K$ be a Galois $p$-extension of totally real number fields. Then, the factor equivalence class of $E_{L}$ as a $\mathbb{Z}[G_{L/K}]$-lattice is uniquely determined by the quotients of $p$-class numbers of subfields of $L$ associated with the pairs $(H/B,\Theta)$ in Theorem $\ref{Bouc's theorem}$.
\end{prop}

\begin{proof}
The quotients involving $w_{L^{H'}}$ in $(\ref{Bouctheorem2})$ are equal to $1$ by Lemma \ref{constant}. The claim follows because we have $\lambda(H')=1$.
\end{proof}

The following corollary is helpful in studying the existence of Minkowski units in Galois $p$-extensions of $\mathbb{Q}$.

\begin{coro}\label{pextensions2}
Let $p$ be an odd prime, and let $L/\mathbb{Q}$ be a Galois $p$-extension. Then $E_{L}$ is factor equivalent to $\mathcal{A}_{G_{L/\mathbb{Q}}}$ as $\mathbb{Z}[G_{L/\mathbb{Q}}]$-lattices if and only if we have 
\begin{equation}\label{Bouc-type relation}
   \underset{B \leq H' \leq H}{\prod}\mathfrak{h}_{L^{H'}}^{n_{H'}} = \underset{B \leq H' \leq H}{\prod}h_{L^{H'}}^{n_{H'}} = \underset{B \leq H' \leq H}{\prod}|H'|^{-n_{H'}} = \underset{B \leq H' \leq H}{\prod}|H'/B|^{-n_{H'}}
\end{equation}
for all the pairs $(H/B, \Theta)$ consisting of a subquotient $H/B$ of $G_{L/\mathbb{Q}}$ and an $H/B$-relation $$\Theta = \sum_{B \leq H' \leq H} n_{H'}\,(H'/B)$$ as in Theorem \ref{Bouc's theorem}.
\end{coro}

For Galois $p$-extensions of totally real $p$-rational number fields, Proposition \ref{pextensions} can be refined as follows.

\begin{prop}\label{application-Bouc}
Let $L/K$ be a Galois $p$-extension of totally real $p$-rational number fields. Then the factor equivalence class of $E_{L}$ as a $\mathbb{Z}[G_{L/K}]$-lattice is determined by the quotients of $p$-class numbers of subfields of $L$ associated to the pairs $(H/B, \Theta)$ in Theorem \ref{Bouc's theorem} such that $H/B \simeq (\mathbb{Z}/p\mathbb{Z})^2$.
\end{prop}

\begin{proof}
By Theorem \ref{Bouc's theorem} and Proposition \ref{pextensions}, it suffices to prove that $G_{L/K}$ has no subquotient $H/B$ isomorphic to the Heisenberg group of order $p^{3}$. By Corollary \ref{totallyrealprational} and Proposition \ref{Demushkin}, $G_{L/K}$ is isomorphic to a subquotient of a Demu\v{s}kin group of rank $2$. It is known that open subgroups of a Demu\v{s}kin group of rank $2$ are themselves Demu\v{s}kin of rank $2$ (cf. \cite[\S 4.5]{Serre-GC}). Hence, every subquotient of a Demu\v{s}kin group of rank $2$ is powerful. The claim follows since the Heisenberg group of order $p^3$ is not powerful.
\end{proof}

Now, we give a proof of Theorem \ref{maintheorem2}. Let $F$ be a non-abelian $p$-rational $p$-extension of $\mathbb{Q}$, and let $F^{\ab}$ be the maximal subfield of $F$ that is abelian over $\mathbb{Q}$. By group theory, $F^{\ab}/\mathbb{Q}$ is not cyclic. Since $\mathbb{Q}$ is $p$-rational, there exists a non-$p$ prime $q$ such that $\Ram(F/\mathbb{Q})=\{p,q\}$. Hence, $F$ is contained in $\mathbb{Q}_{S_p \cup \{q\}}$, and the arithmetic of $F$ can be analyzed via the tower $\mathbb{Q}_{S_p \cup \{q\}}/\mathbb{Q}$.

The case where $p$ splits in $F$ is easy and can be settled immediately.

\begin{lemm}
    If $p$ splits in $F$, then $F$ does not admit a local Minkowski unit.
\end{lemm}

\begin{proof}
It is well known that if $F$ admits a local Minkowski unit, then so does every subfield of $F$ that is Galois over $\mathbb{Q}$. By the Burnside basis theorem, $p$ splits in $F$ if and only if it splits in $\mathbb{Q}_{\{q\}}$. In this case, Theorem \ref{genus equiv for genus extensions} (ii) together with Remark \ref{criterionofBurns} implies that $F$ does not admit a local Minkowski unit.
\end{proof}

 Hence, it remains to prove Theorem \ref{maintheorem2} in the case where $p$ does not split in $F$. In what follows, we work under this assumption. Following \S \ref{uniquepprime}, we denote by $p_K$ the unique $p$-adic prime of each subfield $K$ of $F$ (both $K$ and $F$ being contained in $\mathbb{Q}_{S_p \cup \{q\}}$). By Corollary \ref{pextensions2} and Proposition \ref{application-Bouc}, $E_{F}$ is factor equivalent to $\mathcal{A}_{G_{F/\mathbb{Q}}}$ if and only if the equality $(\ref{Bouc-type relation})$ holds for every extension $K_{p,q}^{\el}/K$ contained in $F$. For ease of notation, we set
\begin{equation*}
\mathcal{I}_{K,p} := \text{the inertia subgroup of } G_{K^{\el}_{p,q}/K} \text{ at } p_{K^{\el}_{p,q}}, 
\qquad
\mathcal{I}_{K,q} := \text{the inertia subgroup of } G_{K^{\el}_{p,q}/K} \text{ at } q_{K^{\el}_{p,q}}.
\end{equation*}

\begin{lemm}\label{lemma-application of Bouc}
Let $F$ be a $p$-rational non-cyclic $p$-extension of $\mathbb{Q}$ with a unique $p$-adic prime. Let $q$ be the rational non-$p$ prime with $\Ram(F/\mathbb{Q})=\{p,q\}$. Let $K$ be a subfield of $F$ with $K_{p,q}^{\el} \subseteq F$. The following claims are valid:
\begin{enumerate}
\item[(i)] We have $\mathcal{I}_{K,p} \neq 1$ and $\mathcal{I}_{K,q} \simeq \mathbb{Z}/p\mathbb{Z}$.
\item[(ii)] If we have $\mathcal{I}_{K,p}=G_{K^{\el}_{p,q}/K}$, then the necessary condition $(\ref{Bouc-type relation})$ associated to $K_{p,q}^{\el}/K$ is not satisfied.
\item[(iii)] If we have $\mathcal{I}_{K,p} \neq G_{K^{\el}_{p,q}/K}$ and $\mathcal{I}_{K,p} \neq \mathcal{I}_{K,q}$, then the necessary condition $(\ref{Bouc-type relation})$ associated to $K_{p,q}^{\el}/K$ is satisfied.
\item[(iv)] If we have $\mathcal{I}_{K,p} = \mathcal{I}_{K,q}$, then the necessary condition $(\ref{Bouc-type relation})$ associated to $K^{\el}_{p,q}/K$ is satisfied if and only if $\mathfrak{h}_{K^{\el}_{p,q}}=\mathfrak{h}_K$.
\end{enumerate}
\end{lemm}

\begin{proof}
\begin{enumerate}
\item[(i)] The subgroup $\mathcal{I}_{K,p}$ must be non-trivial because otherwise $p_K$ splits in $K_{p,q}^{\el}$. Since $K$ is $p$-rational and $G_{K^{\el}_{p,q}/K}$ is not cyclic, we have $\mathcal{I}_{K,q} \neq 1$. The cyclicity of $\mathcal{I}_{K,q}$ follows from the class field theory. 
\item[(ii)] The Galois group $G_{K_{p,q}^{\el}/K}$ is the maximal elementary abelian quotient of $G_{S_{p} \cup \{q_{K}\}}(K)$. If $\mathcal{I}_{K,p}=G_{K^{\el}_{p,q}/K}$, then $G_{S_{p} \cup \{q_K\}}(K)$ coincides with the inertia subgroup at the unique $p$-adic prime by the Burnside basis theorem. Hence, every subfield of $F$ containing $K$ has $p$-class number $1$ by Proposition \ref{pclassnumber-inertiasubgroup} and Corollary \ref{totally ramified}. Therefore, the quotient of $p$-class groups in $(\ref{Bouc-type relation})$ associated with $K^{\el}_{p,q}/K$ is $1$, and so the equality in $(\ref{Bouc-type relation})$ fails.
\item[(iii)] The group $\mathcal{I}_{K,p}$ is non-trivial by (i). By the $p$-rationality of $K$, the subgroup $\mathcal{I}_{K,q}$ corresponds to the first layer $K_1$ of $K_{\infty}/K$. Let $K'$ be the subfield of $K_{p,q}^{\el}$ fixed by the subgroup $\mathcal{I}_{K,p}$. By the assumption, the extensions $K^{\el}_{p,q}/K'$ and $K_1/K$ are ramified precisely at the $p$-adic primes. Hence, we have $$\mathfrak{h}_K=\mathfrak{h}_{K_1}=\mathfrak{h}_{K'}=\mathfrak{h}_{K^{\el}_{p,q}}=1$$ by Lemma \ref{p-class number cyclic} (ii). For any other degree $p$-extension $N$ of $K$ in $K^{\el}_{p,q}$ distinct from $K'$ and $K_1$, the extension $K^{\el}_{p,q}/N$ is unramified, since $K_{p,\mathfrak{q}}^{\el}$ is equal to both $NK_1$ and $NK'$. By Lemma \ref{p-class number cyclic} (i), we then have $\mathfrak{h}_{N}=p$ for all such $N$. Thus, condition $(\ref{Bouc-type relation})$ is satisfied for the extension $K_{p,q}^{\el}/K$.
\item[(iv)] In this case, the first layer $K_1$ of $K_{\infty}/K$ is the subfield of $K^{\el}_{p,q}$ corresponding to the subgroup $\mathcal{I}_{K,p}=\mathcal{I}_{K,q}$. For any degree-$p$ extension $N$ of $K$ contained in $K^{\el}_{p,q}$ other than $K_1$, we have $K^{\el}_{p,q}=N_1$. Therefore, we have
\begin{equation*}
    \mathfrak{h}_{K} = p \cdot \mathfrak{h}_{K_1}, \qquad \mathfrak{h}_N=p \cdot\mathfrak{h}_{K^{\el}_{p,q}}
\end{equation*}
for all such $N$ by Proposition \ref{pclassnumber}. From these identities, the claim follows.

%Let $\mathfrak{H}'$ be the set of degree $p$ extensions of $K$ in $K^{\el}_{p,q}$ other than $K_1$. By Proposition \ref{pclassnumber}, we have $\mathfrak{h}_K=ph_{K_1,p}$ and $h_{M,p}= ph_{K^{\el}_{p,q},p}$ for all $M \in \mathfrak{H}'$. From these identities, we can check the claim.
\end{enumerate}
\end{proof}

Before proving Theorem \ref{maintheorem2}, we recall a well-known fact about the subgroup lattice of the non-abelian semi-direct product $\mathbb{Z}/p^{2}\mathbb{Z} \rtimes \mathbb{Z}/p\mathbb{Z}$.

\begin{lemm}\label{subgrouplattice}
Let $p$ be an odd prime, and let $G \simeq \mathbb{Z}/p^2\mathbb{Z} \rtimes \mathbb{Z}/p\mathbb{Z}$ be the non-abelian semidirect product. Then, $G$ has a unique subgroup $H$ isomorphic to $(\mathbb{Z}/p\mathbb{Z})^2$. Moreover, every subgroup of $G$ of order $p$ is contained in $H$.  
\end{lemm}

\begin{proof}
By elementary arguments, one checks that the center $Z(G)$ of $G$ coincides with the commutator subgroup $[G,G]$. Every subgroup of $G$ of order $p^2$ must contain the center, because otherwise $G$ would be abelian. Since the quotient $G/Z(G)=G/[G,G]$ is isomorphic to $(\mathbb{Z}/p\mathbb{Z})^2$, there are precisely $p+1$ subgroups of $G$ of order $p^2$. It is known that $G$ admits the following presentation (cf. \cite[Exercise 5.3.6]{robinson})
\begin{equation*}
    \langle \, x , y \, \mid \, x^{p^2}=1=y^p, \, y^{-1}xy=x^{1+p} \, \rangle.
\end{equation*}
Using the congruence $$(xy)^n \equiv x^ny^n(y^{-1}x^{-1}yx)^{n(n-1)/2} \equiv [[G,G],G]$$ (cf. \cite[\S 0.1]{analytic}), we can check that the elements $xy^{i}$ for $1 \leq i \leq p$ generate $p$ distinct normal cyclic subgroups of order $p^2$. Therefore, the non-cyclic subgroup $H$ of order $p^2$ generated by $x^p$ and $y$ is the unique subgroup of $G$ isomorphic to $(\mathbb{Z}/p\mathbb{Z})^2$. The last claim follows because every subgroup of order $p$ is either equal to $Z(G)$ or generates a rank-$2$ elementary abelian subgroup with $Z(G)$. 
\end{proof}

\begin{proof}[Proof of Theorem \ref{maintheorem2}]
First, the extension $F/F^{\ab}$ is cyclic and $q_{F^{\ab}}$ is totally ramified in $F$. If $F/F^{\ab}$ were not cyclic, then we would have $(F^{\ab})_{p,q}^{\el} \subseteq F$. Then, we have a contradiction to the maximality of $F^{\ab}$ because $(F^{\ab})_1$, which is a subfield of $(F^{\ab})^{\el}_{p,q}$, is abelian over $\mathbb{Q}$. Similarly, $q_{F^{\ab}}$ must be totally ramified in $F$ because otherwise $(F^{\ab})_1$ would be a subfield of $F$.

Since $G_{F/F^{\ab}}$ is cyclic, there exists a unique extension $W$ of $F^{\ab}$ in $F$ with $[W:F^{\ab}]=p$. As both $F$ and $F^{\ab}$ are Galois over $\mathbb{Q}$, the same holds true for $W$. Let $N$ be the subfield of $F^{\ab}$ such that $G_{F^{\ab}/N}$ is isomorphic to $(\mathbb{Z}/p\mathbb{Z})^{2}$. Then, $G_{W/N}$ is either isomorphic to $\mathbb{Z}/p^2\mathbb{Z} \times \mathbb{Z}/p\mathbb{Z}$ or to the non-abelian semi-direct product $\mathbb{Z}/p^2\mathbb{Z} \rtimes \mathbb{Z}/p\mathbb{Z}$. By considering their subgroup lattices, one finds a subfield $N \subsetneq W' \subsetneq F^{\ab}$ with $G_{W/W'} \simeq (\mathbb{Z}/p\mathbb{Z})^{2}$ (cf. Lemma \ref{subgrouplattice}). We will show that the necessary condition $(\ref{Bouc-type relation})$ associated with $W/W'$ for the factor equivalence of $E_{F}$ and $\mathcal{A}_{G_{F/\mathbb{Q}}}$ is not satisfied.

\begin{equation*}
\begin{tikzcd}
                                                                                                                      &                                                                                                                         & F \\
                                                                                                                      & W \arrow[ru, no head]                                                                                                   &   \\
F^{\mathrm{ab}} \arrow[ru, no head] \arrow[ru, "\mathbb{Z}/p\mathbb{Z}", no head, bend left=49]                       &                                                                                                                         &   \\
                                                                                                                      & W' \arrow[uu, no head] \arrow[lu, no head] \arrow[uu, "(\mathbb{Z}/p\mathbb{Z})^2" description, no head, bend right=49] &   \\
N \arrow[uu, no head] \arrow[ru, no head] \arrow[uu, "(\mathbb{Z}/p\mathbb{Z})^2" description, no head, bend left=49] &                                                                                                                         &  
\end{tikzcd}
\end{equation*}

Since $F^{\ab}$ contains $\mathbb{Q}^{\el}_{p,q}$, the prime $q$ is ramified in $F$. Moreover, by the primitivity of $S_p \cup \{q\}$, we have $q \not\equiv 1 \pmod{p^2}$. Hence, by local class field theory, the ramification index of $q$ in $F^{\ab}/\mathbb{Q}$ is exactly $p$. 

The prime $q$ must also be ramified in $W'/\mathbb{Q}$, for otherwise $F^{\ab}/W'$ would be ramified at $q_{W'}$, 
forcing $\mathcal{I}_{W',q} = G_{W/W'}$ and thereby contradicting Lemma \ref{lemma-application of Bouc} (i). 

Since the compositum $\mathbb{Q}_{\{q\}}W'$ is abelian over $\mathbb{Q}$, the previous argument shows that 
$q_{W'}$ is unramified in $W'\mathbb{Q}_{\{q\}}/W'$. 
Thus we deduce that $\mathbb{Q}_{\{q\}} \subset H_{W'}$.

%We claim that the ramification index of $q$ in $F^{\ab}/\mathbb{Q}$ is equal to $|q-1|_p$. By class field theory, the ramification index can be at most $|q-1|_p$. If it is strictly smaller than $|q-1|_p$, then $F^{\ab}\mathbb{Q}_{\{q\}}/\mathfrak{F}$ is totally ramified at $q_{\mathfrak{F}}$ for some intermediate field $F^{\ab} \subseteq \mathfrak{F} \subsetneq F^{\ab}\mathbb{Q}_{\{q\}}$. In particular, $\mathfrak{F}_{p,q}^{\el}$ is abelian over $\mathbb{Q}$. Since $W$ is a subfield of $\mathfrak{F}_{p,q}^{\el}$, we have a contradiction with the maximality of $F^{\ab}$. The prime $q_{W'}$ is unramified in the extension $F^{\ab}/W'$ because $T_{W',q}$ is cyclic and $W/F^{\ab}$ is ramified at $q_{F^{\ab}}$. Therefore, the ramification index of $q$ in $W'$ is also equal to $|q-1|_p$. Since $\mathbb{Q}_{\{q\}}W'$ is abelian over $\mathbb{Q}$, the abelian extension $\mathbb{Q}_{\{q\}}W'/W'$ is unramified. Hence, $\mathbb{Q}_{\{q\}}$ is a subfield of $H_{W'}$.
\vskip 5pt

We now complete the proof by applying Lemma \ref{lemma-application of Bouc} to the extension $W/W'$. There are three possible cases for the relation among $\mathcal{I}_{W',p}, \mathcal{I}_{W',q}$, and $G_{W/W'}$.

\begin{itemize}
\item[(i)] If we have $\mathcal{I}_{W',p}=G_{W/W'}$, then $E_F$ is not factor equivalent to $\mathcal{A}_{G_{F/\mathbb{Q}}}$ by Lemma \ref{lemma-application of Bouc} (ii).
\item[(ii)] The case $\mathcal{I}_{W',p} \neq G_{W/W'}$ and $\mathcal{I}_{W',p} \neq \mathcal{I}_{W',q}$ (i.e., Lemma \ref{lemma-application of Bouc} (iii)) cannot occur. Otherwise, we have $\mathfrak{h}_{W'}=1$ by the proof of Lemma \ref{lemma-application of Bouc} (iii). Since $H_{W'}=W'$ contains $\mathbb{Q}_{\{q\}}$, we then obtain $G_{W/W'}=\mathcal{I}_{W',p}$ by Corollary \ref{totally ramified}, a contradiction.

\item[(iii)] If we have $\mathcal{I}_{W',p} \neq G_{W/W'}$ and $\mathcal{I}_{W',p}=\mathcal{I}_{W',q}$, then the necessary condition $(\ref{Bouc-type relation})$ associated to $W/W'$ is satisfied if and only if we have $\mathfrak{h}_{W}=\mathfrak{h}_{W'}$ by Lemma \ref{lemma-application of Bouc} (iv). Furthermore, the proof of Lemma \ref{lemma-application of Bouc} (iv) shows that this equality holds if and only if we have $\mathfrak{h}_{N}=p\mathfrak{h}_{W'}$ for every degree-$p$ extension $N$ of $W'$ in $W$ in which both $p_{W'}$ and $q_{W'}$ are ramified. Since we have $\mathbb{Q}_{\{q\}} \subseteq H_{W'}$, we have $\mathfrak{h}_{W'}=\mathfrak{h}_{N}$ for all such $N$ by Proposition \ref{cyclic pq ramified}. Thus, the necessary condition is not satisfied.
\end{itemize}
\end{proof}

\section{Relative Galois module structure of the unit lattices of totally real $p$-rational number fields}\label{relGalmod}

In this final section, we study the relative Galois module structure of the unit lattices for Galois extensions of totally real $p$-rational number fields. We begin by introducing the following notation, which is convenient for the discussion. For each Galois extension $L/K$ of number fields and each finite set $S$ of places of $K$, we write $S_{L}$ for the set of places of $L$ above $S$, and $S_{L,f}$ for the set of finite places of $L$ above $S$. We denote by $\mathcal{O}_{L,S}^{\times}$ the group $\mathcal{O}_{L,S_L}^{\times}$ of $S_L$-units of $L$, and by $\mathrm{Cl}_S(L)$ the $S_L$-ideal class group of $L$. Finally, we set $E_{L,S}:=\mathcal{O}^{\times}_{L,S}/\mu(L)$.

In \cite{Burns4}, for a \textit{fixed} general finite group $G$ and \textit{varying} Galois extensions $L/K$ of number fields with $G_{L/K} \simeq G$, Burns studied the $\mathbb{Z}_p[G]$-module structures of the pro-$p$ completions of several arithmetic objects attached to $L$. In particular, his results apply to $\mathbb{Z}_p \otimes_{\mathbb{Z}} E_{L,S}$ for any finite set $S$ of primes of $K$ containing $S_{p} \cup \Ram(L/K)$. In this section, for every Galois extension $L/K$ with $G_{L/K} \simeq G$, we fix a group isomorphism and consider $E_L$ as a $\mathbb{Z}[G]$-lattice.

This investigation of the $\mathbb{Z}_p[G]$-structure of $E_{L,S}$ for varying $L/K$ is of intrinsic interest because, as $L/K$ varies, the $\mathbb{Z}_{p}$-ranks of $\mathbb{Z}_p \otimes_{\mathbb{Z}} E_{L,S}$ are unbounded. Consequently, if the $p$-Sylow subgroup of $G$ is not cyclic of order $1$, $p$, or $p^{2}$, then infinitely many non-isomorphic indecomposable $\mathbb{Z}_{p}[G]$-lattices can appear in the Krull-Schmidt decomposition of $\mathbb{Z}_p \otimes_{\mathbb{Z}} E_{L,S}$ (cf. \cite[\S 81.A]{curtis1966representation}). Beyond this intrinsic interest, the knowledge of relative Galois module structures of the unit lattices also has significant applications in the study of tamely ramified pro-$p$ extensions of number fields (cf. \cite{hajir2021deficiency, LeeLim}).

Let $S$ be a finite set of places of $\mathbb{Q}$ containing $S_{p}$, and let $L/K$ be a Galois extension of number fields with Galois group $G$ that is unramified outside $S$. In \cite{Burns4}, Burns proved that the sum of the $\mathbb{Z}_p$-ranks of the non-projective components in a Krull-Schmidt decomposition of $\mathbb{Z}_p \otimes_{\mathbb{Z}} E_{L,S}$ (as $\mathbb{Z}_p[G]$-lattices) is bounded above by a function that depends only on $|G|$, the $p$-rank of $\mathrm{Cl}_S(L(\zeta_p))$, and $|S_{L,f}|$. By the Jordan-Zassenhaus theorem (cf. \cite[Thm. 24.2]{methodsrepre}), we obtain the following result.

\begin{theo}(\cite[Cor. 4.1]{Burns4})\label{Burnscohomology}
Let $S$ be a finite set of places of $\mathbb{Q}$ containing $S_{p}$. Let $G$ be a finite group and $b$ a natural number. Define $\mathrm{Ext}(G,S, b)$ to be the family of Galois extensions of number fields $L/K$ satisfying the following conditions :
\begin{enumerate}
    \item[(i)] $G_{L/K} \simeq G$,
    \item[(ii)] $\zeta_{p} \in K$,
    \item[(iii)] $L/K$ is unramified outside $S$,
    \item[(iv)] $\mathrm{rk}_{p}(\textrm{Cl}_{S}(L)) + |S_{L,f}| \leq b$.
\end{enumerate}
Then, there exists a finite set $\Omega$ of $\mathbb{Z}_p[G]$-lattices such that for every $L/K \in \mathrm{Ext}(G, S, b)$, there exists $X \in \Omega$ and a projective $\mathbb{Z}_p[G]$-lattice $P$ with $$\mathbb{Z}_p \otimes_{\mathbb{Z}} E_{L,S} \simeq X \oplus P$$ as $\mathbb{Z}_p[G]$-lattices.
\end{theo}

\begin{rema}
Theorem \ref{Burnscohomology} was obtained by analyzing the Krull-Schmidt decomposition of \'{e}tale cohomology groups and the compactly supported $p$-adic \'{e}tale cohomology groups of general $p$-adic Galois representations over number fields. Theorem \ref{Burnscohomology} is formulated in terms of $S$-units for $S \supset S_p$ because $\mathbb{Z}_p \otimes_{\mathbb{Z}} \mathcal{O}^{\times}_{L,S}$ is isomorphic to the \'{e}tale cohomology group $H^{1}(\mathrm{Spec}(\mathcal{O}_{L,S})_{\acute{e}t}, \mathbb{Z}_{p}(1))$. The paper \cite{Burns4} also treats the ray class groups (cf. Artin-Verdier duality) and higher algebraic K-groups (cf. Voevodsky's Theorem). For further details we refer the readers to \cite[\S 4]{Burns4}.
\end{rema}

We conclude this paper by establishing that for Galois extensions of totally real $p$-rational number fields, an analogous phenomenon occurs in the relative Galois module structure of the group of \textit{ordinary} units.

\begin{theo}[Theorem \ref{maintheorem3}]\label{maintheorem3'}
Let $G$ be a finite group and $p$ an odd prime. Then, there exists a finite set $\Omega$ of $\mathbb{Z}_p[G]$-lattices such that for every Galois extension $L/K$ of totally real $p$-rational number fields with $G_{L/K} \simeq G$, there exists $X \in \Omega$ and an integer $m \geq 0$ such that $\mathbb{Z}_p \otimes_{\mathbb{Z}} E_{L}$ is factor equivalent to $X \oplus \mathbb{Z}_{p}[G]^{m}$ as $\mathbb{Z}_{p}[G]$-lattices.
\end{theo}

\begin{proof}
By Lemma \ref{p-class number cyclic} (i), for every extension $M/N$ of totally real $p$-rational number fields, we have the inequalities 
\begin{equation*}
    v_{p}(\mathfrak{h}_{N})-v_p([M:N]) \leq v_p(\mathfrak{h}_{M}) \leq v_p(\mathfrak{h}_{N}) + \log_p [M:N].
\end{equation*}
 Therefore, for any fixed $G$-relation $\Theta$ and varying Galois extensions $M/N$ of totally real $p$-rational number fields with $G_{M/N} \simeq G$, only finitely many values for $v_p(\mathcal{C}_{\Theta}(E_M))$ can occur by Lemma \ref{constant}.

Fix a $\mathbb{Z}$-basis $\Upsilon$ of the group of $G$-relations, and let $\Xi$ be the set of functions $f:\Upsilon \to \mathbb{Z}$ such that there exists a Galois extension $M/N$ of totally real $p$-rational number fields with $G_{M/N} \simeq G$, such that $$f(\Theta) = v_p(\mathcal{C}_{\Theta}(E_{M}))$$ for every $\Theta \in \Upsilon$. By the above argument, $\Xi$ is a finite set.

For each $x \in \Xi$, choose a Galois extension $L_x/K_x$ of totally real $p$-rational fields with minimal $[K_x:\mathbb{Q}]$ such that we have $G_{L_x/K_x} \simeq G$ and $x(\Theta)=v_p(\mathcal{C}_{\Theta}(E_{L_x}))$ for every $\Theta \in \Upsilon$. We now show that the set $$\Omega = \big \{ \,\mathbb{Z}_p \otimes_{\mathbb{Z}} E_{L_x}\, \big \}_{x \in \Xi}$$ satisfies the claim of the theorem.

Let $L/K$ be a Galois extension of totally real $p$-rational fields with $G_{L/K} \simeq G$. Proposition \ref{factor-regulatorconstant ppart} implies that the factor equivalence class of $\mathbb{Z}_p \otimes_{\mathbb{Z}} E_{L}$ as a $\mathbb{Z}_{p}[G]$-lattice is uniquely determined by $v_p(\mathcal{C}_{\Theta}(E_L))$ for all the $G$-relations $\Theta \in \Upsilon$.
From the construction of $\Xi$, there exists a unique $y \in \Xi$ such that we have $v_p(\mathcal{C}_{\Theta}(E_{L})) = y(\Theta)$ for every $\Theta \in \Upsilon$. We claim that $\mathbb{Z}_p \otimes_{\mathbb{Z}} E_L$ is factor equivalent to $\mathbb{Z}_p \otimes_{\mathbb{Z}} V$ as $\mathbb{Z}_p[G]$-lattices, where we have
\begin{equation*}
    V:=E_{L_y} \oplus \mathbb{Z}[G]^m, \qquad \text{with} \,\,m=\frac{[L:\mathbb{Q}] - [L_y : \mathbb{Q}]}{|G|}.
\end{equation*}

%$V$ denotes the $\mathbb{Z}[G]$-lattice $E_{L_y} \oplus \mathbb{Z}[G]^{m}$ with $m=([L:\mathbb{Q}] - [L_y : \mathbb{Q}])/|G|$.

By the Dirichlet-Herbrand theorem (cf. \cite[Lem. I.3.6]{Gras3}), the two $\mathbb{Z}[G]$-lattices $E_{L}$ and $V$ have the same rational representation. Moreover, for every $\Theta \in \Upsilon$, we have
\begin{equation*}
v_p(\mathcal{C}_{\Theta}(V)) = v_p(\mathcal{C}_{\Theta}(E_{L_y})) = v_p(\mathcal{C}_{\Theta}(E_L)).
\end{equation*}
The first equality follows from Lemma \ref{directproduct} and Lemma \ref{cyclic subgroups}. The second equality follows from the construction of $E_{L_{y}}$. Therefore, by Proposition \ref{factor-regulatorconstant ppart}, we conclude that $\mathbb{Z}_p \otimes_{\mathbb{Z}} E_L$ is factor equivalent to $\mathbb{Z}_p \otimes_{\mathbb{Z}} V$ as $\mathbb{Z}_p[G]$-lattices. 
\end{proof}

As discussed in \S \ref{uniquepprime}, the arithmetic is particularly simple in a pro-$p$ tower of totally real $p$-rational fields $F_{S_p \cup \{\mathfrak{q}\}}/F$ for $(F,p, \mathfrak{q})$ satisfying \eqref{condU}.

\begin{theo}[Theorem \ref{maintheorem4}]
Let $F$ be a totally real $p$-rational number field, and let $\mathfrak{q}$ be a non-$p$-adic prime of $F$ such that $(F,p, \mathfrak{q})$ satisfies \eqref{condU}. For every Galois extension $L/K$ satisfying
\begin{equation*}
F_{\{\mathfrak{q}\}} \subseteq K \subseteq L \subset F_{S_p \cup \{\mathfrak{q}\}},  
\end{equation*}
the $\mathbb{Z}[G_{L/K}]$-lattices $E_{L}$ and $\mathcal{A}_{G_{L/K}} \oplus I_{G_{L/K}} \oplus \mathbb{Z} \oplus \mathbb{Z}[G_{L/K}]^{[K:\mathbb{Q}]-2}$ are factor equivalent.
\end{theo}

\begin{proof}
Let $L/K$ be a Galois extension of totally real $p$-rational fields as in the statement of the theorem. From the isomorphisms $\mathbb{Q}[G_{L/K}] \simeq \mathbb{Q} \otimes_{\mathbb{Z}} (I_{G_{L/K}} \oplus \mathbb{Z})$ of $\mathbb{Q}[G_{L/K}]$-modules, we have an isomorphism
\begin{equation*}
    \mathbb{Q} \otimes_{\mathbb{Z}} E_{L} \simeq \mathbb{Q} \otimes_{\mathbb{Z}} (\mathcal{A}_{G_{L/K}} \oplus I_{G_{L/K}} \oplus \mathbb{Z} \oplus \mathbb{Z}[G_{L/K}]^{[K:\mathbb{Q}]-2})
\end{equation*}
of $\mathbb{Q}[G_{L/K}]$-modules. By Theorem \ref{factor-regulatorconstant}, it remains to show that $E_{L}$ and $\mathcal{A}_{G_{L/K}} \oplus I_{G_{L/K}} \oplus \mathbb{Z} \oplus \mathbb{Z}[G_{L/K}]^{[K:\mathbb{Q}]-2}$ have the same regulator constants for all $G$-relations of $G_{L/K}$. By the Brauer-Kuroda theorem and Theorem \ref{Bartel}, for a $G$-relation $\Theta = \sum_{H \leq G_{L/K}} n_H H$, we have
\begin{equation*}
    \mathcal{C}_{\Theta}(E_L) = \mathcal{C}_{\Theta}(\mathbb{Z}) \cdot \underset{H \leq G_{L/K}}{\prod}\bigg ( \frac{w_{L^H}}{h_{L^H}\lambda(H)} \bigg)^{2n_H}.
\end{equation*}
As a consequence, we have $\mathcal{C}_{\Theta}(E_L)=\mathcal{C}_{\Theta}(\mathbb{Z})$ by the argument used in the proof of Proposition \ref{pextensions} and Proposition \ref{pclassnumber-inertiasubgroup}. On the other hand, we have
\begin{equation*}
    \mathcal{C}_{\Theta}(\mathcal{A}_{G_{L/K}} \oplus I_{G_{L/K}} \oplus \mathbb{Z} \oplus \mathbb{Z}[G_{L/K}]^{[K:\mathbb{Q}]-2}) = \mathcal{C}_{\Theta}(\mathbb{Z})
\end{equation*}
by Lemma \ref{directproduct}, Lemma \ref{cyclic subgroups}, and Corollary \ref{inverse}.
\end{proof}

\begin{rema}
We remark that our results can be extended to other families of Galois extensions provided that the Galois group and its relationship with inertia subgroups at the ramified primes are sufficiently simple.

One easy example is when $F$ is an imaginary number field with a unique $p$-adic prime and $p$-class number $1$. Then $G_{S_p}(F)$ coincides with the inertia subgroup at the $p$-adic place. Hence, the $p$-adic prime of $F$ is totally ramified in $F_{S_p}$. As a consequence, every extension of $F$ in $F_{S_p}$ has $p$-class number $1$.

Thus Theorem \ref{maintheorem2} holds for every extension $L$ of $F$ in $F_{S_p}$ that is Galois over an imaginary quadratic field and has non-cyclic $G_{L/F}$. Moreover, Theorem \ref{maintheorem4} also applies to every Galois extension of number fields in the tower $F_{S_p}/F$, provided $F$ does not contain the $p$th roots of unity. 

As an explicit illustration, for $p=3$ one can take $F=\mathbb{Q}(\sqrt{6}, \sqrt{-1})$ \cite[p.~239]{Wingberglg}. We note that this field is not $3$-rational.
\end{rema}

\begin{rema}
After this work was completed, the results of Burns in the direction of Theorem~\ref{Burnscohomology} on the $S$-unit group were extended to ordinary unit lattices in \cite[Thm.~A]{KumonLim}. Since the classification of integral representations is unavailable in general, information on the $\mathbb{Z}_p$-ranks of the non-projective components alone does not effectively bound the number of possible $\mathbb{Z}_p[G]$-module structures (cf. \cite[Thm. B]{KumonLim}).

Theorem~\ref{maintheorem3} shows that the number of possible factor equivalence classes of unit lattices of totally real $p$-rational fields—depending only on the set of $G$-relations—is highly restricted, in sharp contrast with genus equivalence. This rigidity also indicates, however, that factor equivalence itself provides only limited information about the underlying Galois module structure.

For readers interested in the Krull–Schmidt decomposition of unit lattices in cyclic $p$-extensions of totally real $p$-rational fields, we refer to \cite{BLM}. In a different direction, Ozaki obtained a result \cite{Ozaki2025} on the Galois structure of unit lattices in cyclic $p$-extensions, indicating that the range of genus equivalence classes of the unit lattice is essentially as large as that of integral $\mathbb{Z}_p[G]$-modules.
\end{rema}

\bibliographystyle{plain}
\bibliography{refs.bib}

\vskip 30pt

\noindent École Supérieure de l'Education et de la Formation Oujda 60000, Morocco \\
Email address: z.bouazzaoui@ump.ac.ma

\vskip 30pt

\noindent Institute of Mathematical Sciences, Ewha Womans University, Seoul 03760, Republic of Korea \\
Email address: donghyeokklim@gmail.com

\end{document}